\documentclass[12pt,a4paper]{amsart}
\usepackage{amsxtra,amssymb}
\usepackage[cmtip,arrow]{xy}
\usepackage{pb-diagram,pb-xy}

\calclayout

\newcommand{\+}{\nobreakdash-}
\renewcommand{\:}{\colon}
\renewcommand{\.}{\mskip.5\thinmuskip}

\newcommand{\rarrow}{\longrightarrow}

\newcommand{\birarrow}{\rightrightarrows}
\newcommand{\ot}{\otimes}
\newcommand{\st}{\star}

\DeclareMathOperator{\Hom}{Hom}
\DeclareMathOperator{\Ext}{Ext}
\DeclareMathOperator{\Tor}{Tor}
\DeclareMathOperator{\coker}{coker}
\DeclareMathOperator{\im}{im}
\DeclareMathOperator{\id}{id}
\DeclareMathOperator{\pd}{pd}

\newcommand{\fR}{{\mathfrak R}}

\newcommand{\boR}{{\mathbb R}}
\newcommand{\boL}{{\mathbb L}}
\newcommand{\Z}{{\mathbb Z}}
\newcommand{\Q}{{\mathbb Q}}

\newcommand{\D}{{\mathcal D}}

\renewcommand{\b}{{\mathsf{b}}}
\newcommand{\co}{{\mathsf{co}}}
\newcommand{\ctr}{{\mathsf{ctr}}}
\newcommand{\abs}{{\mathsf{abs}}}
\newcommand{\inj}{{\mathsf{inj}}}
\newcommand{\proj}{{\mathsf{proj}}}
\newcommand{\fl}{{\mathsf{fl}}}
\newcommand{\ttors}{{\mathsf{tors}}}
\newcommand{\rrcot}{{\mathsf{rcot}}}
\newcommand{\sop}{{\mathsf{op}}}

\newcommand{\Hot}{{\mathsf{Hot}}}

\newcommand{\sA}{{\mathsf A}}
\newcommand{\sC}{{\mathsf C}}
\newcommand{\sD}{{\mathsf D}}
\newcommand{\sT}{{\mathsf T}}
\newcommand{\sF}{{\mathsf F}}

\newcommand{\vect}{{\operatorname{\mathsf{--vect}}}}
\newcommand{\modl}{{\operatorname{\mathsf{--mod}}}}
\newcommand{\tors}{{\operatorname{\mathsf{-tors}}}}
\newcommand{\ctra}{{\operatorname{\mathsf{-ctra}}}}
\newcommand{\secmp}{{\operatorname{\mathsf{-secmp}}}}
\newcommand{\adj}{{\operatorname{\mathsf{-adj}}}}

\newcommand{\bu}{{\text{\smaller\smaller$\scriptstyle\bullet$}}}
\newcommand{\lrarrow}{\.\relbar\joinrel\relbar\joinrel\rightarrow\.}
\newcommand{\llarrow}{\.\leftarrow\joinrel\relbar\joinrel\relbar\.}

\newcommand{\Section}[1]{\bigskip\section{#1}\medskip}

\theoremstyle{plain}
\newtheorem{thm}{Theorem}[section]

\newtheorem{lem}[thm]{Lemma}
\newtheorem{prop}[thm]{Proposition}
\newtheorem{cor}[thm]{Corollary}
\theoremstyle{definition}

\newtheorem{rem}[thm]{Remark}

\begin{document}

\title{Triangulated Matlis equivalence}

\author{Leonid Positselski}

\address{Department of Mathematics, Faculty of Natural Sciences,
University of Haifa, Mount Carmel, Haifa 31905, Israel; and
\newline\indent Laboratory of Algebraic Geometry, National Research
University Higher School of Economics, Moscow 117312; and
\newline\indent Sector of Algebra and Number Theory, Institute for
Information Transmission Problems, Moscow 127051, Russia}

\email{posic@mccme.ru}

\begin{abstract}
 This paper is a sequel to~\cite{Pmgm} and~\cite{Pcta}.
 We extend the classical Harrison--Matlis module category
equivalences to a triangulated equivalence between the derived
categories of the abelian categories of torsion modules and
contramodules over a Matlis domain.
 This generalizes to the case of any commutative ring $R$ with a fixed
multiplicative system $S$ such that the $R$\+module $S^{-1}R$ has
projective dimension~$1$.
 The latter equivalence connects complexes of $R$\+modules with
$S$\+torsion and $S$\+contramodule cohomology modules.
 It takes a nicer form of an equivalence between the derived categories
of abelian categories when $S$ consists of nonzero-divisors or
the $S$\+torsion in $R$ is bounded.
\end{abstract}

\maketitle

\tableofcontents

\section*{Introduction}
\medskip

 The category of torsion abelian groups is abelian.
 So is the category of reduced cotorsion abelian groups~\cite{Pcta};
in fact, this observation was made, in a much greater generality,
already in~\cite[Remarks in~\S2]{Mat} (the reader should be warned
that Matlis, following Harrision~\cite{Harr}, calls such groups simply
``cotorsion'', including the ``reduced'' condition into the definition
of ``cotorsion'').
 The derived categories of these two abelian categories are
equivalent: one has
\begin{equation} \label{abelian-torsion-redcotorsion}
 \sD^\st(\Z\modl_\ttors)\simeq\sD^\st(\Z\modl_\rrcot)
\end{equation}
for any derived category symbol $\st=\b$, $+$, $-$, or~$\varnothing$.

 Furthermore, the abelian category of torsion abelian groups is
the Cartesian product of the abelian categories of $p$\+primary
torsion abelian groups, while the abelian category of reduced
cotorsion abelian groups is the Cartesian product of the abelian
categories of $p$\+contramodule abelian groups.
 The equivalence of categories~\eqref{abelian-torsion-redcotorsion}
is the Cartesian product of the equivalences
\begin{equation} \label{abelian-torsion-contra}
 \sD^\st(\Z\modl_{p\tors})\simeq\sD^\st(\Z\modl_{p\ctra}).
\end{equation}

 The latter equivalence, in turn, sits in the intersection of two
classes of triangulated equivalences.
 On the one hand, there is an equivalence between the coderived
category of discrete modules and the contraderived category of
contramodules over a pro-coherent topological ring $\fR$ with
a dualizing complex~$\D^\bu$ \cite[Section~D.2]{Pcosh}.
 In particular, given a Noetherian commutative ring $R$ with a fixed
ideal $I\subset R$ and a dualizing complex of $I$\+torsion $R$\+modules
$D^\bu$, there is an equivalence between the coderived category
of $I$\+torsion $R$\+modules modules and the contraderived category
of $I$\+contramodule $R$\+modules~\cite[Section~C.1]{Pcosh}
\begin{equation} \label{dualizing-complex-torsion-modules}
 \sD^\co(R\modl_{I\tors})\simeq\sD^\ctr(R\modl_{I\ctra}).
\end{equation}
 The equivalence is provided by the derived functors of $\Hom$ and
tensor product with the dualizing complex~$D^\bu$.

 The simplest example of
the equivalence~\eqref{dualizing-complex-torsion-modules} occurs
when $I$ is a \emph{maximal} ideal in~$R$.
 Then an injective envelope $C$ of the irreducible $R$\+module $R/I$,
viewed as a one-term complex of $R$\+modules, is a dualizing complex of
$I$\+torsion $R$\+modules (see the discussions
in~\cite[Section~5]{Yek0} and~\cite[Remark~4.10]{Pmgm}).
 Taking the $\Hom$ from $C$ and the tensor product with $C$ establishes
a (covariant) equivalence between the additive categories of injective
$I$\+torsion $R$\+modules and projective $I$\+contramodule $R$\+modules,
\begin{equation} \label{maximal-ideal-proj-inj-equivalence}
 \Hom_R(C,{-})\:\,R\modl_{I\tors}^\inj\.\simeq\. R\modl_{I\ctra}^\proj\!:\,
 C\ot_R{-}.
\end{equation}

 For any Noetherian commutative ring $R$ with an ideal $I$,
the coderived category $\sD^\co(R\modl_{I\tors})$ is equivalent to
the homotopy category of unbounded complexes of injective $I$\+torsion
$R$\+modules and the contraderived category $\sD^\ctr(R\modl_{I\ctra})$
is equivlalent to the homotopy category of unbounded complexes
of projective $I$\+contramodule $R$\+modules,
\begin{equation} \label{co-contra-derived-homotopy}
 \sD^\co(R\modl_{I\tors})\simeq\Hot(R\modl_{I\tors}^\inj)
 \ \.\text{and}\ \.
 \sD^\ctr(R\modl_{I\ctra})\simeq\Hot(R\modl_{I\ctra}^\proj).
\end{equation}
 Combinining~\eqref{co-contra-derived-homotopy}
with~\eqref{maximal-ideal-proj-inj-equivalence}, we obtain
the equivalence of exotic derived
categories~\eqref{dualizing-complex-torsion-modules} in the case
of a maximal ideal $I\subset R$.

 On the other hand, for any commutative ring $R$ with a finitely
generated ideal $I\subset R$ there is a natural equivalence between
the full subcategories of the derived category $\sD^\st(R\modl)$
formed by the complexes with $I$\+torsion and $I$\+contramodule
cohomology modules~\cite[Section~3]{Pmgm}
\begin{equation} \label{non-weakly-proregular}
 \sD^\st_{I\tors}(R\modl)\simeq\sD^\st_{I\ctra}(R\modl).
\end{equation}
 When the finitely generated ideal $I\subset R$ is \emph{weakly
proregular} (in particular, when the ring $R$ is Noetherian),
the equivalence~\eqref{non-weakly-proregular} takes a nicer form
of an equivalence between the derived categories of the abelian
categories of $I$\+torsion and $I$\+contramodule
$R$\+modules~\cite[Sections~1\+-2]{Pmgm}
\begin{equation} \label{weakly-proregular-equivalence}
 \sD^\st(R\modl_{I\tors})\simeq\sD^\st(R\modl_{I\ctra}).
\end{equation}
 Notice the difference between
the equivalences~\eqref{dualizing-complex-torsion-modules}
and~\eqref{weakly-proregular-equivalence} in that
the former is an equivalence between the coderived and
the contraderived categories, while the latter connects
the (bounded or unbounded) conventional derived categories.
 The results of~\cite[Sections~4\+-5]{Pmgm} extend
the equivalence~\eqref{weakly-proregular-equivalence} to the case of
the absolute derived categories $\sD^{\abs+}$, \,$\sD^{\abs-}$,
or~$\sD^\abs$.

 The equivalence~\eqref{non-weakly-proregular} and,
especially, \eqref{weakly-proregular-equivalence} is provided by
the functors of $\Hom$ and tensor product with a \emph{dedualizing
complex} of $R$\+modules~$B^\bu$.
 The simplest example occurs when $I$ is a \emph{principal} ideal
in~$R$.
 Then the dedualizing complex $B^\bu$ for the ideal $I$ is
\begin{equation} \label{one-element-dedualizing}
 R\lrarrow R[s^{-1}],
\end{equation}
where $s$ is a generating element of an ideal~$I$.
 More precisely, the weak proregularity condition in the case of
a principal ideal $I$ simply says that the $s$\+torsion in $R$
should be bounded.
 Assuming this, one can use a complex of $I$\+torsion $R$\+modules
quasi-isomorphic to~\eqref{one-element-dedualizing} in the role
of a dedualizing complex~$B^\bu$.

 To compare the dualizing and dedualizing complexes, one can express
them in terms of the derived functor $\boR\Gamma_I$ of the functor
$\Gamma_I(M)=\varinjlim_n\Hom_R(R/I^n,M)$ of maximal $I$\+torsion
submodule of an $R$\+module~$M$.
 Let $D_R^\bu$ be a dualizing complex of $R$\+modules.
 Then one can take
$$
 D^\bu=\boR\Gamma_I(D_R^\bu)\quad\text{and}\quad
 B^\bu=\boR\Gamma_I(R).
$$

 In the case of a regular Noetherian ring $R$ (of finite Krull
dimension), the abelian categories $R\modl_{I\tors}$ and
$R\modl_{I\ctra}$ have finite homological dimension, so there is
no difference between the conventional derived, co/contraderived,
and the absolute derived category for them.
 There is also no difference between a dualizing and a dedualizing
complex.
 So the two triangulated
equivalences~\eqref{dualizing-complex-torsion-modules}
and~\eqref{weakly-proregular-equivalence} become one and the same.
 We have explained that the equivalence of derived
categories~\eqref{abelian-torsion-contra} is a common particular case
of~\eqref{dualizing-complex-torsion-modules}
and~\eqref{weakly-proregular-equivalence}.

 Back in the 1950--60's, when the derived categories were not yet known,
a version of the equivalence~\eqref{abelian-torsion-redcotorsion} was
first observed by Harrison~\cite{Harr}.
 In fact, there are \emph{two} equivalences of additive categories
in~\cite[Section~2]{Harr}, both of which we now see as related to
the triangulated equivalence~\eqref{abelian-torsion-redcotorsion}.
 One of them is an equivalence between what we would now call
the additive categories of injective objects in $\Z\modl_\ttors$ and
projective objects in $\Z\modl_\rrcot$ \cite[Proposition~2.1]{Harr}.
 It is provided by the functors $\Hom_\Z(\Q/\Z,{-})$ and
$\Q/\Z\ot_\Z{-}$.
 The other one is an equivalence between the full subcategory of
objects having no injective direct summands in $\Z\modl_\ttors$ and
the full subcategory of objects having no projective direct summands
in $\Z\modl_\rrcot$ \cite[Proposition~2.3]{Harr} (see
also~\cite[Theorem~55.6]{Fuc}).
 This one is provided by the functors $\Ext^1_\Z(\Q/\Z,{-})$ and
$\Tor_1^\Z(\Q/\Z,{-})$.

 Harrison then remarks that one can obtain a correspondence between
arbitrary objects in $\Z\modl_\ttors$ and $\Z\modl_\rrcot$ by
taking a direct sum of the $\Hom$ and $\Ext^1$ from $\Q/\Z$ on
the one side and a direct sum of the~$\ot$ and $\Tor_1$ with
$\Q/\Z$ on the other side.
 This is no longer an equivalence of categories, of course, but only
a bijection between the isomorphism classes of objects.
 The discussion of this ``nonnatural isomorphism'' continues
in Matlis'~\cite[Remarks in~\S3]{Mat}, where he observes that such
a bijection between the torsion and reduced cotorsion modules holds
over any Dedekind domain, but not over other domains.
 One feels pained by reading today these discussions which would be
so much illuminated and clarified by an introduction of the derived
category point of view.

 In the modern homological language, we say that the equivalence of
derived categories~\eqref{abelian-torsion-redcotorsion} is provided
by the derived functors $\boR\Hom_\Z(\Q/\Z,{-})$ and
$\Q/\Z\ot_\Z^\boL{-}$, which have homological dimension~$1$.
 So, generally speaking, acting in each direction, they take a group
into a two-term complex of groups.
 Restricting the equivalence~\eqref{abelian-torsion-redcotorsion}
to those the complexes concentrated in the cohomological degree~$0$ 
on each side which are taken by this equivalence to complexes
concentrated in the cohomological degree~$0$ on the other side,
one obtains the equivalence of categories $\Z\modl_\ttors^\inj\simeq
\Z\modl_\ttors^\proj$ \cite[Proposition~2.1]{Harr}.
 Restricting the equivalence~\eqref{abelian-torsion-redcotorsion}
to those complexes of reduced cotorsion groups that are
concentrated in the cohomological degree~$0$ and are taken
to complexes of torsion groups concentrated in the cohomological
degee~$-1$ (and vice versa), one obtains the second Harrison's
equvalence~\cite[Proposition~2.3]{Harr}.

 Furthermore, the abelian categories $\Z\modl_\ttors$ and
$\Z\modl_\rrcot$ also have homological dimension~$1$, hence every
complex in these categories is noncanonically isomorphic to
the direct sum of its cohomology groups.
 Decomposing the two-term complexes of torsion groups into the direct
sums of their cohomology groups, one recovers Harrison's direct sum
decomposition of reduced cotorsion groups into the ``torsion-free''
and ``adjusted'' parts~\cite[Proposition~2.2]{Harr}.

 Matlis~\cite{Mat} extended Harrison's theory to modules over arbitrary
commutative domains.
 The two Matlis' equivalences of categories, generalizing the two
Harrison's equivalences, are~\cite[Theorems~3.4 and~3.8]{Mat}
(see also~\cite[Theorem~VIII.2.8]{FS}).
 The aim of this paper is to interpret the two Matlis equivalences of
additive categories of modules as a single triangulated equivalence
between the derived categories.

 There is one caveat: our generality level differs from Matlis'.
 On the one hand, Matlis' paper~\cite{Mat} only deals with integral
domains~$R$.
 In the language of~\cite{Mat}, an element $x$ in an $R$\+module $M$
is said to be torsion if there exists $r\in R$, \ $r\ne0$ such that
$rx=0$.
 Subsequently, in the book~\cite[Chapters~I\+-II]{Mat2}, Matlis
extends his theory to arbitrary commutative rings.
 In the context of~\cite{Mat2}, an element $x\in M$ is said to be
torsion if there exists a nonzero-divisor $r\in R$ such that
$rx=0$.

 This generalizes naturally much further: let $R$ be an arbitrary
commutative ring and $S\subset R$ be a multiplicative set.
 We say that an element $x$ in an $R$\+module $M$ is
\emph{$S$\+torsion} if there exists $s\in S$ such that $sx=0$.
 For the beginning, one can assume that all the elements of $S$ are
nonzero-divisors in~$R$ (cf.~\cite{GM} and~\cite{AHT}).
 Then this restriction can be relaxed to the condition that
the $S$\+torsion in $R$ is bounded, or even dropped altogether.
 One of the aims of this paper is to explain how to extend
the classical theory to the situation when $S$ contains some
zero-divisors.

 On the other hand, the category of what Matlis calls ``cotorsion
$R$\+modules'' (and we call \emph{$S$\+contramodule $R$\+modules})
is only well-behaved homologically (i.~e., abelian with an exact
embedding functor $R\modl_{S\ctra}\rarrow R\modl$) when the projective
dimension of the $R$\+module $S^{-1}R$ does not exceed~$1$.
 Matlis observes in~\cite[\S10]{Mat} that ``A remarkable smoothing of
the whole theory takes place under the assumption'' of the projective
dimension of the field of fractions $Q$ of his domain $R$ being
equal to~$1$.
 Still, he formulates the main results, including the equivalences of
categories in~\cite[\S3]{Mat}, for an arbitrary commutative domain.

 Commutative domains $R$ for which $\pd_RQ=1$ are now called
\emph{Matlis domains}~\cite[Section~2]{Mat1}, \cite[Section~IV.4]{FS}.
 In this paper, we pay tribute to Matlis' (and now traditional)
generality preferences by discussing the \emph{$S$\+topology} for
an arbitrary multiplicative subset $S$ in a commutative ring~$R$, but
then make the assumption $\pd_RS^{-1}R\le1$ in order to formulate and
prove our homological results.

 This assumption holds, in particular, for every countable
multiplicative subset $S\subset\nopagebreak R$.
 It also holds under certain more complicated countability conditions
(\cite[Theorem~3.2]{FS0} and~\cite[Theorem~1.1]{AHT}), and under
Noetherian Krull one-dimensionality conditions (\cite[Theorem~4.2]{Mat2}
and~\cite[Section~13]{Pcta}).
 In fact, one has $\pd_RS^{-1}R\le1$ for every multiplicative subset $S$
in a Noetherian commutative ring of Krull dimension~$1$
(\cite[Corollaire~II.3.3.2]{RG} and~\cite[Remark~13.9]{Pcta}).

 It remains to explain the connection between the \emph{MGM}
(\emph{Matlis--Greenlees--May}) \emph{duality} of the paper~\cite{Pmgm}
and the triangulated Matlis equivalence/duality of the present paper.
 To pass from the former to the latter, one first restricts generality
in the situation of a finitely generated ideal $I$ in a commutative
ring $R$, by assuming that $I$ is a principal ideal generated by
an element $s\in R$.
 Then one expands generality in a different direction, by replacing
the multiplicative set $\{s^n\mid n\in\Z_{\ge0}\}\subset R$ by
an arbitrary multiplicative subset $S\subset R$.

 Returning to the above discussion of dedualizing complexes, let us
point out that, in our present context, the two-term complex
\begin{equation} \label{matlis-dedualizing}
 R\lrarrow S^{-1}R,
\end{equation}
or a complex of $S$\+torsion $R$\+modules quasi-isomorphic to it,
plays the role of a dedualizing complex (cf.\ the more
elementary~\eqref{one-element-dedualizing}).
 When all the elements of $S$ are nonzero-divisors in $R$,
the complex~\eqref{matlis-dedualizing} is quasi-isomorphic to
the $R$\+module $(S^{-1}R)/R$.
 In this connection, it is worth mentioning that there is a long
tradition of considering the functor $\Ext^1_R(Q/R,{-})$ in
the homological algebra of commutative domains $R$, going back to
the papers of Nunke, Harrison, and Matlis~\cite{Nun,Harr,Mat1,Mat}.

 To conclude this introduction, let us explain, in the most simple
homological terms, why the equivalence of triangulated
categories~\eqref{abelian-torsion-redcotorsion} holds.
 First of all, both $\sD^\st(\Z\modl_\ttors)$ and
$\sD^\st(\Z\modl_\rrcot)$ are full triangulated subcategories
in $\sD^\st(\Z\modl)$.
 Furthermore, the derived category of $\Q$\+vector spaces
$\sD^\st(\Q\vect)$, which is also a full triangulated subcategory
in $\sD^\st(\Z\modl)$, takes part in two semiorthogonal decompositions
of the category $\sD^\st(\Z\modl)$.
 The left orthogonal complement to $\sD^\st(\Q\vect)$ in
$\sD^\st(\Z\modl)$ is $\sD^\st(\Z\modl_\ttors)$, while the right
orthogonal complement to $\sD^\st(\Q\vect)$ in the same ambient
category is $\sD^\st(\Z\modl_\rrcot)$.
 As a corollary to these two semiorthogonal decompositions, one has
$$
 \sD^\st(\Z\modl_\ttors)\simeq\sD^\st(\Z\modl)/\sD^\st(\Q\vect)
 \simeq\sD^\st(\Z\modl_\rrcot).
$$

 Similarly in the equivalence~\eqref{abelian-torsion-contra},
both $\sD^\st(\Z\modl_{p\tors})$ and $\sD^\st(\Z\modl_{p\ctra})$ are
full triangulated subcategories in $\sD^\st(\Z\modl)$.
 So is the derived category $\sD^\st(\Z[p^{-1}]\modl)$ of abelian
groups with invertible action of~$p$.
 Furthermore, the full triangulated subcategory
$\sD^\st(\Z[p^{-1}]\modl)\subset\sD^\st(\Z\modl)$ takes part in two
semiorthogonal decompositions of the ambient category
$\sD^\st(\Z\modl)$.
 The left orthogonal complement to $\sD^\st(\Z[p^{-1}]\modl)$ in
$\sD^\st(\Z\modl)$ is $\sD^\st(\Z\modl_{p\tors})$, while the right
orthogonal complement is $\sD^\st(\Z\modl_{p\ctra})$.
 In the result, one obtains~\cite[Corollary~3.5]{Pmgm}
$$
 \sD^\st(\Z\modl_{p\tors})\simeq\sD^\st(\Z\modl)/
 \sD^\st(\Z[p^{-1}]\modl)\simeq\sD^\st(\Z\modl_{p\ctra}).
$$

 As we show in this paper, these results generalize to an arbitrary
commutative ring $R$ with a multiplicative subset $S\subset R$ such
that the projective dimension of the $R$\+module $S^{-1}R$ does not
exceed~$1$.
 In this setting, the full subcategory of $S$\+contramodule
$R$\+modules $R\modl_{S\ctra}\subset R\modl$ is an abelian category
with an exact embedding functor $R\modl_{S\ctra}\rarrow R\modl$.
 The full subcategory of $S$\+torsion $R$\+modules
$R\modl_{I\tors}\subset R\modl$ is always so (in fact, $R\modl_{I\tors}$
is even a Serre subcategory, which $R\modl_{S\ctra}$ is
not~\cite[Section~1]{Pcta}).

 Let $\sD^\st_{S\tors}(R\modl)$ and $\sD^\st_{S\ctra}(R\modl)\subset
\sD^\st(R\modl)$ denote the full subcategories of complexes with
$S$\+torsion and, respectively, $S$\+contramodule cohomology modules.
 The derived category of $(S^{-1}R)$\+modules $\sD^\st((S^{-1}R)\modl)$
is a full triangulated subcategory in $\sD^\st(R\modl)$ taking part
in two semiorthogonal decompositions of $\sD^\st(R\modl)$.
 The left orthogonal complement to $\sD^\st((S^{-1}R)\modl)$ in
$\sD^\st(R\modl)$ coincides with $\sD^\st_{S\tors}(R\modl)$, while
the right orthogonal complement is $\sD^\st_{S\ctra}(R\modl)$.
 Hence the triangulated equivalences
\begin{equation} \label{torsion-contra-cohomology-modules}
 \sD^\st_{S\tors}(R\modl)\simeq\sD^\st(R\modl)/
 \sD^\st((S^{-1}R)\modl)\simeq\sD^\st_{S\ctra}(R\modl).
\end{equation}

 Furthermore, when the $S$\+torsion in the ring $R$ is bounded,
the triangulated functors $\sD^\st(R\modl_{S\tors})\rarrow
\sD^\st(R\modl)$ and $\sD^\st(R\modl_{S\ctra})\rarrow\sD^\st(R\modl)$
are fully faithful.
 Their essential images coincide with the full subcategories
$\sD^\st_{S\tors}(R\modl)$ and $\sD^\st_{S\ctra}(R\modl)\subset
\sD^\st(R\modl)$.
 So we have
\begin{equation} \label{essential-images}
 \sD^\st_{S\tors}(R\modl)\simeq\sD^\st(R\modl_{S\tors})
 \quad\text{and}\quad
 \sD^\st_{S\ctra}(R\modl)\simeq\sD^\st(R\modl_{S\ctra})
\end{equation}
 Comparing~\eqref{torsion-contra-cohomology-modules}
with~\eqref{essential-images}, we obtain a triangulated equivalence
between the derived categories of the abelian categories
$R\modl_{S\tors}$ and $R\modl_{S\ctra}$
\begin{equation} \label{equivalence-between-derived}
 \sD^\st(R\modl_{S\tors})\simeq\sD^\st(R\modl_{S\ctra}).
\end{equation}
 In the last section of this paper, we show that the triangulated
equivalences~\eqref{equivalence-between-derived} hold for the absolute
derived categories $\sD^{\abs+}$, \,$\sD^{\abs-}$, and $\sD^\abs$ of
$S$\+torsion modules and $S$\+contramodule modules over a commutative
ring $R$ with bounded $S$\+torsion and $\pd_SS^{-1}R\le1$ as well as
for the conventional derived categories
$\sD^\b$, \,$\sD^+$, \,$\sD^-$, and~$\sD$.

 It should be mentioned that a very different generalization of
the Harrison--Matlis additive category equivalences was developed
many years ago by Facchini in~\cite{Fac1,Fac2}
(see also~\cite[Example~13.4]{GT}).
 This was further generalized and formulated in terms of fully
faithful/Verdier quotient functors between triangulated categories
by Bazzoni~\cite{Baz}.
 Let us briefly point out one of the differences between our
approaches.
 In Facchini's papers, the aim was to study the additive categories
of divisible and reduced modules.
 The restriction to torsion modules was viewed as undesirable and
successfully removed.
 In the present paper, our aim is to study the abelian categories of
$S$\+torsion modules and their covariantly dual counterparts,
the $S$\+contramodules. {\uchyph=0\par}

\medskip

\textbf{Acknowledgment.}
 This paper owes its existence to Jan Trlifaj, who invited me
to Prague, told me about the $R$\+topology being a traditional topic
in the commutative algebra of non-Noetherian domains, and showed me
the book~\cite{FS}.
 I~am also grateful to the anonymous referee for several helpful
suggestions.
 The author's research is supported by the Israel Science
Foundation grant~\#\,446/15.

\Section{Preliminaries}  \label{preliminaries-secn}

 Let $R$ be an associative ring.
 We denote by $R\modl$ the abelian category of left $R$\+modules.
 A pair of full subcategories $(\sT,\sF)$ in $R\modl$ is called
a \emph{torsion theory}~\cite{Dic} if one has $\Hom_R(T,F)=0$ for all
$T\in\sT$ and $F\in\sF$, and for every left $R$\+module $M$
there exists a short exact sequence $0\rarrow T\rarrow M\rarrow F
\rarrow0$ with $T\in\sT$ and $F\in\sF$.
 In this case, such a short exact sequence is unique and functorial.
 Given an $R$\+module $M$, the $R$\+module $T$ is the maximal
submodule of $M$ belonging to $\sT$, and the $R$\+module $F$ is
the maximal quotient module of $M$ belonging to~$\sF$.

 The full subcategory $\sT\subset R\modl$ is called the \emph{torsion
class} of a torsion theory $(\sT,\sF)$, and the full subcategory
$\sF\subset R\modl$ is called the \emph{torsion-free class}.
 For any torsion theory $(\sT,\sF)$ in $R\modl$, the torsion class
$\sT$ is closed under the passages to arbitrary quotient objects,
extensions, and infinite direct sums, while the torsion-free class
$\sF$ is closed under the passages to subobjects, extensions, and
infinite products in $R\modl$.
 Since the direct sum of a family of modules is a submodule of their
product, it follows that the class $\sF$ is closed under infinite
direct sums, too.

 Conversely, any full subcategory $\sT\subset R\modl$ closed under
quotient objects, extensions, and infinite direct sums is
the torsion class of a certain torsion theory $(\sT,\sF)$, and any full
subcategory $\sF\subset R\modl$ closed under subobjects, extensions,
and infinite products is the torsion-free class of a torsion theory
$(\sT,\sF)$. 
 The complementary class can be uniquely recovered by the rules that
$\sF$ consists of all the $R$\+modules $F$ such that $\Hom_R(T,F)=0$
for all $T\in\sT$, and $\sT$ consists of all the $R$\+modules $T$
such $\Hom_R(T,F)=0$ for all $F\in\sF$.

 A torsion theory $(\sT,\sF)$ is called \emph{hereditary} if the class
$\sT\subset R\modl$ is closed under quotient objects.
 In this case, $\sT$ is a Serre subcategory in $R\modl$; so, in
particular, it is an abelian category with an exact embedding
functor $\sT\rarrow R\modl$.

\medskip

 From now on and for the rest of this paper, let $R$ be a commutative
ring and $S\subset R$ be a multiplicative subset.
 An element $x\in M$ in an $R$\+module $M$ is said to be
\emph{$S$\+torsion} if there exists $s\in S$ such that $sx=0$ in $M$.
 The submodule of all $S$\+torsion elements in $M$ is denoted by
$\Gamma_S(M)\subset M$ and the embedding morphism
$\Gamma_S(M)\rarrow M$ is denoted by~$\gamma_{S,M}$.

 We will use the notation $S^{-1}M=S^{-1}R\ot_RM$ for
the $S$\+localization of an $R$\+module~$M$.
 An $R$\+module $M$ is said to be \emph{$S$\+torsion} if
$\Gamma_S(M)=M$, or equivalently, if $S^{-1}M=0$.
 An $R$\+module $M$ is said to be \emph{$S$\+torsion-free} if
$\Gamma_S(M)=0$.
 The full subcategories of $S$\+torsion $R$\+modules and $S$\+torsion
free $R$\+modules form a hereditary torsion theory in $R\modl$;
the related canonical short exact sequence is $0\rarrow \Gamma_S(M)
\rarrow M\rarrow M/\Gamma_S(M)\rarrow0$ for any $R$\+module~$M$.
 In particular, the full subcategory $R\modl_{S\tors}$ of all
$S$\+torsion $R$\+modules is an abelian category with an exact
embedding functor $R\modl_{S\tors}\rarrow R\modl$.

 An $R$\+module $M$ is said to be \emph{$S$\+divisible} if for every
element $s\in S$ the action map $s\:M\rarrow M$ is surjective.
 An $R$\+module $M$ is said to be \emph{$S$\+reduced} if it has no
$S$\+divisible $R$\+submodules.
 The full subcategories of $S$\+divisible and $S$\+reduced
$R$\+modules form a torsion theory in $R\modl$; the $S$\+divisible
modules are the torsion class and the $S$\+reduced $R$\+modules
are the torsion-free class.
 In addition to the general closure properties of such classes,
the class of all $S$\+divisible $R$\+modules is also closed under
infinite products.
 An $R$\+module $M$ is both $S$\+torsion-free and $S$\+divisible if
and only if it is an $(S^{-1}R)$\+module.

 An $R$\+module $M$ is said to be \emph{$S$\+h-divisible} if is
a quotient $R$\+module of an $(S^{-1}R)$\+module.
 Clearly, every $S$\+h-divisible $R$\+module is $S$\+divisible.
 The class of all $S$\+h-divisible $R$\+modules is closed under
quotient objects, infinite direct sums, and infinite products
in $R\modl$, but it is not always closed under extensions.
 Every $R$\+module $M$ has a unique maximal $S$\+h-divisible submodule
$h_S(M)\subset M$, which can be constructed as the image of the natural
map $\Hom_R(S^{-1}R,M)\rarrow M$.

 An $R$\+module $M$ is said to be \emph{$S$\+h-reduced} if it has no
$S$\+h-divisible submodules, or equivalently, if $\Hom_R(S^{-1}R,M)=0$.
 An $R$\+module $M$ is $S$\+h-reduced if and only if $h_S(M)=0$, but
the quotient module $M/h_S(M)$ for an arbitrary $R$\+module $M$ is
not always $S$\+$h$-reduced.
 The class of all $S$\+h-reduced $R$\+modules is closed under
subobjects, extensions, infinite direct sums, and infinite products
in $R\modl$.
 Every $S$\+reduced $R$\+module is $S$\+h-reduced.
 Every $S$\+h-reduced $S$\+torsion-free $R$\+module is $S$\+reduced
(because every $S$\+divisible $S$\+torsion-free $R$\+module is
$S$\+h-divisible).

 So the full subcategories of $S$\+h-divisible and $S$\+h-reduced
$R$\+modules do not form a torsion theory in $R\modl$ in general.
 However, when $\pd_RS^{-1}R\le1$, the problem disappears and these
two classes \emph{do} form a torsion theory, as we will see below in
Lemmas~\ref{h-torsion-theory} and~\ref{torsion-free-and-h-divisible}.
 Furthermore, when all the elements of $S$ are nonzero-divisors in $R$
and $\pd_RS^{-1}R\le1$, the classes of $S$\+divisible and
$S$\+h-divisible $R$\+modules coincide~\cite[Theorem~2.6]{Ham},
\cite[Theorem~3.2]{FS0}, \cite[Proposition~6.4]{AHT}.
 Hence the classes of $S$\+reduced and $S$\+h-reduced $R$\+modules
also coincide.

\begin{lem}
 Let\/ $0\rarrow L\rarrow M\rarrow N\rarrow0$ be a short exact sequence
of $R$\+modules such that the $R$\+module $N$ is $S$\+h-divisible,
while the $R$\+module $L$ is $S$\+torsion-free and $S$\+h-divisible.
 Then the $R$\+module $M$ is $S$\+h-divisible.
\end{lem}

\begin{proof}
 Let $N$ be the quotient $R$\+module of an $(S^{-1}R)$\+module~$D$.
 Pulling back the short exact sequence $0\rarrow L\rarrow M\rarrow
N\rarrow 0$ with respect to the morphism $D\rarrow N$, we obtain
a short exact sequence of $R$\+modules $0\rarrow L\rarrow E\rarrow D
\rarrow0$, where $D$ and $L$ are $(S^{-1}R)$\+modules.
 It follows that $E$ is an $(S^{-1}R)$\+module, too; and $M$ is
a quotient $R$\+module of~$E$.
\end{proof}

\begin{lem} \label{ext-from-s-minus-1-r-module}
 Let $C$ be an $R$\+module and $D$ be an $(S^{-1}R)$\+module.
 Assume that\/ $\Ext^i_R(S^{-1}R,C)=0$ for all\/ $0\le i\le n$, where
$n\ge0$ is a fixed integer.
 Then\/ $\Ext^i_R(D,C)=\nobreak0$ for all\/ $0\le i\le n$.
\end{lem}

\begin{proof}
 Arguing by increasing induction in~$i$, pick a short exact sequence
of $(S^{-1}R)$\+modules $0\rarrow B\rarrow F\rarrow D\rarrow0$ with
a free $(S^{-1}R)$\+module~$F$.
 Then we have an exact sequence \hbadness=1100
$$
 \dotsb\lrarrow \Ext^{i-1}_R(B,C)\lrarrow\Ext_R^i(D,C)\lrarrow
 \Ext_R^i(F,C)\lrarrow\dotsb,
$$
where $\Ext_R^i(F,C)=0$ since $\Ext_R^i(S^{-1}R,C)=0$, and
$\Ext_R^{i-1}(B,C)=0$ by the induction assumption.
\end{proof}

 An $R$\+module $C$ is said to be \emph{$S$\+cotorsion} if
$\Ext^1_R(S^{-1}R,C)=0$, and \emph{strongly $S$\+cotorsion} if
$\Ext^n_R(S^{-1}R,C)=0$ for all $n\ge1$.
 An $R$\+module $C$ is called an \emph{$S$\+contramodule} if
it is $S$\+h-reduced and $S$\+cotorsion, that is
$\Ext^n_R(S^{-1}R,C)=0$ for $n=0$ and~$1$.
 An $R$\+module $C$ is called a \emph{strong $S$\+contramodule}
if it is $S$\+h-reduced and strongly $S$\+cotorsion, that is
$\Ext^n_R(S^{-1}R,C)=0$ for all $n\ge0$.

 When $R$ is a domain and $S=R\setminus\{0\}$, our (strong)
$S$\+contramodules are what are called ``(strongly) cotorsion
$R$\+modules'' in~\cite{Mat}.
 For any commutative ring $R$ with a multiplicative subset $S$,
the full subcategory of $S$\+contramodules in $R\modl$ is
the ``right perpendicular category'' to the $R$\+module $S^{-1}R$,
as defined by Geigle and Lenzing in~\cite[Section~1]{GL}
(cf.\ Theorem~\ref{contramodule-category-thm} below).

\begin{lem} \label{contramodule-short-exact}
 Let\/ $0\rarrow C'\rarrow C\rarrow C''\rarrow0$ be a short exact
sequence of $R$\+modules.  Then \par
\textup{(a)} if $C'$ and $C''$ are $S$\+contramodules, then $C$
is an $S$\+contramodule; \par
\textup{(b)} if $C$ is an $S$\+contramodule, then $C''$ is
$S$\+h-reduced if and only if $C'$ is an $S$\+contramodule; \par
\textup{(c)} if $C$ is an $S$\+contramodule and $C'$ is a strong
$S$\+contramodule, then $C''$ is an $S$\+contramodule.
\end{lem}

\begin{proof}
 This is~\cite[Lemma~1.1]{Mat} or~\cite[Theorem~1.5]{Mat2}.
 The proof is obvious.
\end{proof}

\begin{lem} \label{contramodules-closed-under}
 The full subcategory of $S$\+contramodule $R$\+modules is closed
under the kernels, extensions, infinite products, and projective
limits in $R\modl$.
\end{lem}

\begin{proof}
 This is essentially a particular case of the first assertion
of~\cite[Proposition~1.1]{GL}.
 Closedness with respect to infinite products is obvious, and
closedness under extensions is provided by
Lemma~\ref{contramodule-short-exact}(a).
 Now let $f\:C\rarrow D$ be a morphism of $S$\+contramodule
$R$\+modules; set $I=\im(f)$ and $E=\ker(f)$.
 Then the $R$\+module $I$ is $S$\+h-reduced as a submodule of
an $S$\+h-reduced $R$\+module~$D$.
 Applying Lemma~\ref{contramodule-short-exact}(b) to
the short exact sequence $0\rarrow E\rarrow C\rarrow I\rarrow0$,
we conclude that $E$ is an $S$\+contramodule $R$\+module.
 Finally, the projective limit of any diagram is the kernel of
a certain morphism between infinite products.
\end{proof}

 Following the traditional notation of $K=Q/R$, where $Q$ is
the field of fractions of a commutative domain $R$, we denote by
$K^\bu$ the two-term complex~\eqref{matlis-dedualizing}
$$
 R\lrarrow S^{-1}R,
$$
where the term $R$ sits in the cohomological degree~$-1$ and
the term $S^{-1}R$ in the cohomological degree~$0$.
 The cokernel $H^0(K^\bu)$ of the morphism $R\rarrow S^{-1}R$ will
be denoted simply by $S^{-1}R/R$.

 When all the elements of $S$ are nonzero-divisors in $R$, so
the morphism $R\rarrow S^{-1}R$ is injective, one can
use the quotient module $S^{-1}R/R$ in lieu of the two-term
complex~$K^\bu$.

 We will use the special notation
$$
 \Tor^R_n(K^\bu,M)=H^{-n}(K^\bu\ot_R^\boL M)=H^{-n}(K^\bu\ot_R M)
$$
and
$$
 \Ext_R^n(K^\bu,M)=\Hom_{\sD^\b(R\modl)}(K^\bu,M[n])
$$
for an $R$\+module $M$ (treating $K^\bu$ as if it were a module
rather than a complex).

\begin{lem} \label{tor-with-k}
 For every $R$\+module $M$, one has \par
\textup{(a)} $\Tor^R_n(K^\bu,M)=0=\Ext_R^n(K^\bu,M)$ for all\/ $n<0$;
\par
\textup{(b)} $\Tor^R_n(K^\bu,M)=0$ and\/ $\Ext_R^n(K^\bu,M)=
\Ext_R^n(S^{-1}R,M)$ for all\/ $n>1$; \par
\textup{(c)} $\Tor^R_0(K^\bu,M)=(S^{-1}R/R)\ot_RM$ and\/
$\Ext_R^0(K^\bu,M)=\Hom_R(S^{-1}R/R,M)$; \par
\textup{(d)} $\Tor^R_1(K^\bu,M)=\Gamma_S(M)$.
\end{lem}

\begin{proof}
 All the assertions follow easily from the (co)homology long exact
sequences related to the distinguished triangle
\begin{equation} \label{main-distinguished-triangle}
 R\lrarrow S^{-1}R\lrarrow K^\bu\lrarrow R[1]
\end{equation}
in $\sD^\b(R\modl)$.
\end{proof}

 Warning: it may well happen that $\Tor_1^R(K^\bu,F)\ne0$ for a flat
$R$\+module~$F$. 
 Similarly, one may have $\Ext^1_R(K^\bu,J)\ne0$ for
an injective $R$\+module~$J$.

\medskip

 Following the exposition in~\cite{Mat}
(see also~\cite[Theorem~1.1]{Mat2}), we introduce special
indexing for the three short exact sequences of low-dimensional
$\Tor$ and $\Ext$ related to the distinguished
triangle~\eqref{main-distinguished-triangle}.
 Concerning the $\Tor$, for any $R$\+module $M$ we have
\begin{equation}\tag{I}
 0\lrarrow M/\Gamma_S(M)\lrarrow S^{-1}R\ot_RM\lrarrow
 \Tor^R_0(K^\bu,M)\lrarrow0.
\end{equation}
 Concerning the $\Ext$, for any $R$\+module $C$ we have an exact
sequence
\begin{multline*}
 0\lrarrow\Ext^0_R(K^\bu,C)\lrarrow\Hom_R(S^{-1}R,C)\lrarrow C \\
 \lrarrow \Ext^1_R(K^\bu,C)\lrarrow\Ext^1_R(S^{-1}R,C)\lrarrow0,
\end{multline*}
which can be rewritten in the form of two short exact sequences
\begin{gather}
\tag{II} 0\lrarrow\Ext_R^0(K^\bu,C)\lrarrow\Hom_R(S^{-1}R,C)
\lrarrow h_S(C)\lrarrow 0, \\
\tag{III} 0\lrarrow C/h_S(C)\lrarrow\Ext_R^1(K^\bu,C)\lrarrow
\Ext_R^1(S^{-1}R,C)\lrarrow0.
\end{gather}

 Let us introduce the notation $\Delta_S(C)=\Ext^1_R(K^\bu,C)$,
and denote by~$\delta_{S,C}$ the natural map $C\rarrow\Delta_S(C)$.

\begin{lem} \label{about-contramodules-lemma}
\textup{(a)} An $R$\+module $C$ is an $S$\+contramodule if and only
if the map $\delta_{S,C}\:C\rarrow\Ext^1_R(K^\bu,C)$ is
an isomorphism. \par
\textup{(b)} Any $R$\+module annihilated by the action of
an element from $S$ is a strong $S$\+contramodule.
\end{lem}

\begin{proof}
 Part~(a): from the short exact sequences~(II\+-III) we see that
the equations $\Hom_R(S^{-1}R,C)=0=\Ext_R^1(S^{-1}R,C)$ imply that
the map~$\delta_{S,C}$ is an isomorphism.
 Conversely, if $\delta_{S,C}$~is an isomorphism, then
$\Ext^1_R(S^{-1}R,C)=0$ and $h_S(C)=0$.
 The latter implies that $\Hom_R(S^{-1}R,C)=0$.

 Part~(b): if $rC=0$ for some $r\in R$, then $r\Ext^*_R(M,C)=0$
for every $R$\+module~$M$.
 If $M$ is an $(S^{-1}R)$\+module, then the $R$\+modules $\Ext^*_R(M,C)$
are $(S^{-1}R)$\+modules annihilated by the action of~$r$.
 When $r\in S$, these can only be zero modules.
\end{proof}

 We denote by $\pd_RM$ the projective dimension of an $R$\+module~$M$.

 An $R$\+module $M$ is said to have \emph{bounded $S$\+torsion}
if there exists $r\in S$ such that $r\Gamma_S(M)=\nobreak0$.
 An $R$\+module $M$ is said to have \emph{no $S$\+h-divisible
$S$\+torsion} if the $R$\+module $\Gamma_S(M)$ is $S$\+h-reduced.
 Clearly, every $R$\+module with bounded $S$\+torsion has no
$S$\+divisible $S$\+torsion.
 The following lemma is our version (of the most important special
case) of \cite[Theorem~2.1]{Mat}; see also~\cite[Lemma~6.1 and
Proposition~6.3]{AHT} and~\cite[Theorem~3.5]{AS}.

\begin{lem} \label{delta-produces-contramodule}
\textup{(a)} For any $R$\+module $C$, the $R$\+module\/
$\Ext_R^0(K^\bu,C)$ is an $S$\+contra\-module. \par
\textup{(b)} If there is no $S$\+h-divisible $S$\+torsion in $C$,
then the $R$\+module\/ $\Ext_R^1(K^\bu,C)$ is an $S$\+contramodule.
\par
\textup{(c)} If\/ $\pd_RS^{-1}R\le1$, then for any $R$\+module $C$
the $R$\+module\/ $\Ext_R^1(K^\bu,C)$ is an $S$\+contramodule.
\end{lem}

\begin{proof}
 There is a spectral sequence
$$
 E_2^{pq}=\Ext^p_R(S^{-1}R,\Ext^q_R(K^\bu,C))\Longrightarrow
 E_\infty^{pq}=\mathrm{gr}^p\Ext_R^{p+q}(S^{-1}R\ot_RK^\bu,\>C)=0,
$$
where $\Ext_R^n(S^{-1}R\ot_RK^\bu,\>C)=\Hom_{\sD^\b(R\modl)}
(S^{-1}R\ot_RK^\bu,\>C[n])=0$ for all $n\in\Z$ and all
$C\in R\modl$, because the complex $S^{-1}R\ot_RK^\bu$ is
contractible.
 The differentials are $d_r^{pq}\:E_r^{pq}\rarrow E_r^{p+r,q-r+1}$,
\ $r\ge2$.
 Now all the differentials going through $E_r^{0,0}$ and
$E_r^{1,0}$ vanish for the dimension reasons, so $E_\infty^{0,0}=0=
E_\infty^{1,0}$ implies $E_2^{0,0}=0=E_2^{1,0}$.
 This proves part~(a).
 Furthermore, the only possibly nontrivial differentials going
through $E_r^{0,1}$ and $E_r^{1,1}$ are
$$
 d_2^{0,1}\:E_2^{0,1}\lrarrow E_2^{2,0}\quad\text{and}\quad
 d_2^{1,1}\:E_2^{1,1}\lrarrow E_2^{3,0}.
$$
 When $\pd_RS^{-1}R\le1$, one has $E_2^{pq}=0$ for $p\ge2$.
 When there is no $S$\+h-divisible $S$\+torsion in $C$, one has
$\Ext_R^0(K^\bu,C)=0$ by Lemma~\ref{tor-with-k}(c), because
$S^{-1}R/R$ is an $S$\+h-divisible $S$\+torsion $R$\+module.
 Hence $E_2^{p,0}=0$ for all $p\ge0$.
 In both cases, $E_\infty^{0,1}=0=E_\infty^{1,1}$ implies $E_2^{0,1}=0=
E_2^{1,1}$, proving parts~(b) and~(c).
\end{proof}

 The following lemma is our version of~\cite[Lemma~1.8 and
Theorem~1.9]{Mat2}.

\begin{lem} \label{h-torsion-theory}
\textup{(a)} If\/ $\pd_RS^{-1}R\le 1$, then the class of all
$S$\+h-divisible $R$\+modules is closed under extensions. \par
\textup{(b)} If all the elements of $S$ are nonzero-divisors in $R$
and the class of all $S$\+h-divisible $R$\+modules is closed under
extensions, then\/ $\pd_RS^{-1}R\le 1$.
\end{lem}

\begin{proof}
 Part~(a): it suffices to check that $h_S(M/h_S(M))=0$ for every
$R$\+module~$M$.
 Indeed, according to the short exact sequence~(III), the quotient
module $M/h_S(M)$ is a submodule in $\Ext_R^1(K^\bu,M)$, and
by Lemma~\ref{delta-produces-contramodule}(c), the $R$\+module
$\Ext_R^1(K^\bu,M)$ is $S$\+h-reduced.
 (For another argument, see the paragraph after
Lemma~\ref{torsion-free-and-h-divisible} below.)

 Part~(b): let $E$ be an $R$\+module; we have to prove that
$\Ext_R^2(S^{-1}R,E)=0$.
 Pick an injective $R$\+module $J$ such that $E$ is a submodule in~$J$.
 Then $\Ext_R^2(S^{-1}R,E)=\Ext_R^1(S^{-1}R,J/E)$.
 As any $R$\+module morphism $R\rarrow J$ can be extended to
an $R$\+module morphism $S^{-1}R\rarrow J$, the $R$\+module $J$
is $S$\+h-divisible.
 Hence so is the $R$\+module $L=J/E$.
 Let $0\rarrow L\rarrow M\rarrow S^{-1}R\rarrow0$ be a short exact
sequence of $R$\+modules.
 By assumption, the $R$\+module $M$ has to be $S$\+h-divisible;
so it is a quotient module of an $(S^{-1}R)$\+module~$D$.
 The composition of surjective morphisms $D\rarrow M\rarrow S^{-1}R$
is a morphism of $(S^{-1}R)$\+modules; so it is a split surjection.
 Composing a splitting $S^{-1}R\rarrow D$ with the morphism
$D\rarrow M$, we obtain a splitting of the surjection
$M\rarrow S^{-1}R$.
 Thus $\Ext_R^1(S^{-1}R,L)=0$.
\end{proof}

 In other words, when $\pd_RS^{-1}R\le 1$, the full subcategories of
$S$\+h-divisible and $S$\+h-reduced $R$\+modules form a torsion
theory in $R\modl$.
 The $S$\+h-divisible modules are the torsion class and
the $S$\+h-reduced modules are the torsion-free class.

 The next lemma is quite standard.
 We include it here for the sake of completeness of the exposition.

\begin{lem}
 If a multiplicative subset $S$ in a commutative ring $R$ is
countable, then $\pd_RS^{-1}R\le1$.
\end{lem}

\begin{proof}
 Let $s_1$, $s_2$, $s_3$,~\dots\ be a sequence of elements of $S$
such that every element of $S$ appears at least once in this sequence.
 Denote by $M$ the inductive limit of the sequence of $R$\+module
morphisms
$$
 R\overset{s_1}\lrarrow R\overset{s_1s_2}\lrarrow R\overset{s_1s_2s_3}
 \lrarrow R\lrarrow\dotsb\lrarrow R\overset{s_1\dotsm s_n}\lrarrow
 R\lrarrow\dotsb
$$
 The natural map $R\rarrow M$ has the property that its kernel and
cokernel are $S$\+torsion $R$\+modules.
 Furthermore, the $R$\+module $M$ is $S$\+torsion-free and
$S$\+divisible.
 Hence $M\simeq S^{-1}R$.
 Now the telescope construction of countable filtered inductive limits
provides a two-term free $R$\+module resolution of the $R$\+module
$S^{-1}R$,
$$
 0\lrarrow\bigoplus\nolimits_{n=1}^\infty R\lrarrow
 \bigoplus\nolimits_{n=1}^\infty R\lrarrow S^{-1}R\lrarrow0.
$$
\end{proof}

 The following important lemma is our version
of~\cite[Proposition~2.4]{Mat}.

\begin{lem} \label{rich-functors}
 Let $b\:A\rarrow B$ and $c\:A\rarrow C$ be two $R$\+module
morphisms such that $C$ is an $S$\+contramodule, while\/ $\ker(b)$
is an $S$\+$h$-divisible $R$\+module and\/ $\coker(b)$ is
an $(S^{-1}R)$\+module.
 Then there exists a unique morphism $f\:B\rarrow C$ such that
$c=fb$, i.~e., the triangle diagram $A\rarrow B\rarrow C$
is commutative.
\end{lem}

\begin{proof}
 Any morphism $f'\:B\rarrow C$ such that $f'b=0$ would factorize
as $B\rarrow\coker(b)\rarrow C$, and any morphism $\coker(b)
\rarrow C$ vanishes when $\coker(b)$ is $S$\+h-divisible and
$C$ is $S$\+h-reduced.
 This proves uniqueness.
 To check existence, notice that the composition $\ker(b)\rarrow A
\rarrow C$ vanishes if $\ker(b)$ is $S$\+h-divisible and $C$ is
$S$\+h-reduced.
 Hence we have a short exact sequence $0\rarrow A/\ker(b)\rarrow B
\rarrow\coker(b)\rarrow 0$ and a morpism $\bar b\:A/\ker(b)\rarrow C$.
 The obstruction to extending the morphism~$\bar b$ to
an $R$\+module morphism $B\rarrow C$ lies in the group
$\Ext^1_R(\coker(b),C)$.
 Applying Lemma~\ref{ext-from-s-minus-1-r-module} to the $R$\+modules
$C$ and $D=\coker(b)$ and the integer $n=1$, we conclude that this
$\Ext$ group vanishes.
\end{proof}

\begin{lem} \label{delta-mod-s}
 For any $R$\+module $A$ and every element $s\in S$, the map
$\bar\delta_{S,A}\:A/sA\allowbreak\rarrow\Delta_S(A)/s\Delta_S(A)$ is
an isomorphism.
\end{lem}

\begin{proof}
 Both $A/sA$ and $\Delta_S(A)/s\Delta_S(A)$ are $R/sR$\+modules.
 Furthermore, any $R/sR$\+module $D$ is an $S$\+contramodule by
Lemma~\ref{about-contramodules-lemma}(b).
 In view of the short exact sequence~(III), the morphism
$\delta_{S,A}\:A\rarrow\Delta_S(A)$ and any morphism $A\rarrow D$
satisfy the conditions of Lemma~\ref{rich-functors}, so there exists
a unique morphism $\Delta_S(A)\rarrow D$ making the triangle
diagram $A\rarrow\Delta_S(A)\rarrow D$ commutative.
 Hence the morphism~$\bar\delta_{S,A}$ induces an isomorphism
$$
 \Hom_{R/sR}(\Delta_S(A)/s\Delta_S(A),D)\simeq\Hom_{R/sR}(A/sA,D),
$$
and it follows that $\bar\delta_{S,A}$ is an isomorphism.
\end{proof}

\Section{The $S$-Topology}  \label{s-topology-secn}

 The \emph{$R$\+topology} was introduced by Nunke
in~\cite[Section~6]{Nun} and studied by Matlis in~\cite[\S6]{Mat}
(see also~\cite[Chapter~II]{Mat2}).
 The \emph{$S$\+topology} is discussed in~\cite{GM}
and~\cite[Section~VIII.4]{FS} in the case of a countable
multiplicative subset $S$ consising of nonzero-divisors
(the discussion in~\cite[Chapter~1]{GT} partly avoids
the countability assumption; cf.\ the counterexample
in~\cite[Section~1.2]{GT}).

 As in Section~\ref{preliminaries-secn}, let $R$ be a commutative ring
and $S\subset R$ be a multiplicative subset.
 By the definition, the \emph{$S$\+topology} on an $R$\+module $A$
is the topology with a base of neighborhoods of zero formed by
the submodules $sA\subset A$, where $s\in S$ are arbitrary elements.
 The \emph{$S$\+completion} of an $R$\+module $A$ is defined as
$$
 \Lambda_S(A)=\varprojlim\nolimits_{s\in S} A/sA,
$$
where the partial (pre)order on the set $S$ is defined by the rule that
$s\le t$ for $s$, $t\in S$ if there exists $r\in R$ for which $t=rs$
(this is the reverse inclusion order on the principal ideals in $R$
generated by elements from~$S$).
 There is an obvious natural $R$\+module morphism
$$
 \lambda_{S,A}\:A\lrarrow\Lambda_S(A).
$$
 An $R$\+module $A$ is said to be \emph{$S$\+separated} if the morphism
$\lambda_{S,A}$ is injective.
 An $R$\+module $A$ is said to be \emph{$S$\+complete} if the morphism
$\lambda_{S,A}$ is surjective.
 We will denote the full subcategory of $S$\+separated and $S$\+complete
$R$\+modules by $R\modl_{S\secmp}\subset R\modl$.

\begin{lem} \label{completion-contramodule}
 For every $R$\+module $A$ \par
\textup{(a)} the $R$\+module $\Lambda_S(A)$ is an $S$\+contramodule;
\par
\textup{(b)} there exists a unique $R$\+module morphism
$\beta_{S,A}\:\Delta_S(A)\rarrow\Lambda_S(A)$ making the triangle
diagram $A\rarrow\Delta_S(A)\rarrow\Lambda_S(A)$ commutative.
 Taken together for all the $R$\+modules $A$,
the morphisms~$\beta_{S,A}$ form a morphism of functors
$\beta_S\:\Delta_S\rarrow\Lambda_S$.
\end{lem}

\begin{proof}
 Part~(a) follows from Lemmas~\ref{about-contramodules-lemma}(b)
and~\ref{contramodules-closed-under}.
 Part~(b) is provided by Lemma~\ref{rich-functors} together with
part~(a) and the short exact sequence~(III).
 The second assertion in part~(b) (commutativity of the diagram in
the definition of a natural transformation of functors)
follows from the uniqueness claim in Lemma~\ref{rich-functors}.
\end{proof}

 The following proposition, which is a direct generalization
of~\cite[Lemma~VIII.4.1]{FS} provable by essentially the same method,
may help the reader feel more comfortable.
 We will not use its part~(b) in this paper.

\begin{prop} \label{countable-completion-complete}
\textup{(a)} The $R$\+module $\Lambda_S(A)$ is $S$\+separated for
every $R$\+module~$A$. \par
\textup{(b)} Suppose that the multiplicative set $S$ is countable.
 Then the $R$\+module $\Lambda_S(A)$ is $S$\+separated and
$S$\+complete for every $R$\+module~$A$.
 Moreover, the functor $\Lambda_S$ is left adjoint to the fully
faithful embedding functor $R\modl_{S\secmp}\rarrow R\modl$.
\end{prop}

\begin{proof}
 By the definition, the $R$\+module $\Lambda_S(A)$ is separated and
complete in the \emph{projective limit topology}, where a base of
neighborhoods of zero is formed the kernel submodules
$U_s\subset\Lambda_S(A)$, \ $s\in S$ of the projection maps
$\Lambda_S(A)\rarrow A/sA$.
 In particular, one has $\bigcap_{s\in S}U_s=0$ in $\Lambda_S(A)$.
 Clearly, the inclusions $s\Lambda_S(A)\subset U_s$ hold for all
$s\in S$.
 Hence we have $\bigcap_{s\in S} s\Lambda_S(A)=0$, and part~(a) follows.

 To prove part~(b), let us show that under its assumption
$U_s=s\Lambda_S(A)$ for every $s\in S$.
 Let $1=s_0\le s_1\le s_2\le\dotsb$ be an increasing chain of elements
from $S$ such that for every $r\in S$ there exists $n\ge1$ for which
$r\le s_n$.
 For every $R$\+module $A$ and every sequence of elements $c_0$, $c_1$,
$c_2$~\dots\ in $\Lambda_S(A)$ define the expression
$$
 \sum\nolimits_{n=0}^\infty s_nc_n\.\in\.\Lambda_S(A)
$$
as the limit of finite partial sums in the projective limit topology
of $\Lambda_S(A)$.
 In other words, if the element $c_n\in\Lambda_S(A)$ is represented by
a family of cosets $(c_{n,r}+rA\in A/rA)_{r\in S}$, \ $c_{n,r}\in A$,
then the element $\sum_{n=0}^\infty s_nc_n$ is given by
$(\sum_{n=0}^\infty s_nc_{n,r}+rA)_{r\in S}$, where the sum is essentially
finite modulo $rA$ for every $r\in S$ due to our condition on
the sequence of elements~$s_n$.

 Now set $t_0=1$ and $t_n=s_1\dotsm s_n$ for every $n\ge1$; this
sequence of elements in $S$ also satisfies our conditions.
 Let $c\in\Lambda_S(A)$ be an element represented by a family of
cosets $(c_r+rA)_{r\in S}$.
 Then $c_{t_{n+1}}-c_{t_n}\in t_nA$ for every $n\ge0$.
 Set $a_0=c_{t_1}\in A$ and choose $a_n\in A$ such that
$c_{t_{n+1}}-c_{t_n}=t_na_n$ for every $n\ge1$.
 Then we have
$$
 c=\sum\nolimits_{n=0}^\infty t_n\lambda_{S,A}(a_n)\.\in\.\Lambda_S(A).
$$
 Assume that $c\in U_s$, and choose $n_0\ge1$ such that
$s\le t_{n_0}$ in~$S$.
 Then $\lambda_{S,A}(\sum_{n=0}^{n_0-1} t_na_n)=
\sum_{n=0}^{n_0-1} t_n\lambda_{S,A}(a_n)\in U_s$, hence
$\sum_{n=0}^{n_0-1}t_na_n\in sA$.
 At last, we have
$$
 \sum\nolimits_{n=n_0}^\infty t_n\lambda_{S,A}(a_n)=
 t_{n_0}\sum\nolimits_{n=0}^\infty s_{n_0+1}\dotsm s_{n_0+n}
 \lambda_{S,A}(a_{n_0+n})\.\in\. s\Lambda_S(A),
$$
hence $c\in s\Lambda_S(A)$ (notice that the sequence of elements
$1$, \,$s_{n_0+1}$, \,$s_{n_0+1}s_{n_0+2}$,~\dots\ $\in S$ also satisfies
our conditions, so the sum in the right-hand side converges).

 We have shown that the $S$\+topology on $\Lambda_S(A)$ coincides
with the projective limit topology; so it is separated and complete.
 Finally, we have $\Lambda_S(A)/s\Lambda_S(A)=\Lambda_S(A)/U_s=A/sA$
for every $s\in A$.
 Let $D$ be an $R$\+module belonging to $R\modl_{S\secmp}$.
 Then $D=\varprojlim_{s\in S}D/sD$.
 So, for any $R$\+module $A$, morphisms of $R$\+modules
$A\rarrow D$ correspond bijectively to compatible systems of $R$\+module
morphisms $A\rarrow D/sD$, or, which is the same, $A/sA\rarrow D/sD$.
 The isomorphism $A/sA=\Lambda_S(A)/s\Lambda_S(A)$ now implies
an isomorphism of the $\Hom$ modules
$$
 \Hom_R(A,D)\simeq\Hom_R(\Lambda_S(A),D),
$$
providing the desired adjunction (cf.\ the proof
of~\cite[Theorem~5.8]{Pcta}).
\end{proof}

 Without the countability assumption, a long list of conditions
equivalent to $S$\+completeness of the $R$\+module $\Lambda_S(A)$
for a given $R$\+module $A$, in the case of a domain $R$ and
$S=R\setminus\{0\}$, is presented in~\cite[Theorem~6.8]{Mat}.
 All of them appear to be also equivalent in our generality.
 In the next theorem, we list only some of these equivalent
conditions from~\cite{Mat}, and add a couple of our own.

 As in the proof of Proposition~\ref{countable-completion-complete},
for every element $s\in S$ we denote by $U_s\subset\Lambda_S(A)$
the kernel of the projection $\Lambda_S(A)\rarrow A/sA$.
 Following~\cite{Mat}, we denote for brevity by $\Pi_A$ the infinite
product $\prod_{s\in S}A/sA$.
 The $R$\+module $\Lambda_S(A)$ is naturally a submodule in~$\Pi_A$.
 Given a submodule $N\subset M$ in an $R$\+module $M$, we say that
$N$ is \emph{$S$\+pure} in $M$ if for every $s\in S$ one has
$N\cap sM=sN$.
 One easily notices that $\im(\lambda_{S,A})$ is always
an $S$\+pure submodule in $\Lambda_S(A)$; moreover,
$\lambda_{S,A}^{-1}(s\Lambda_S(A))=sA$ for every $s\in S$.

\begin{thm} \label{when-completion-complete}
 For any commutative ring $R$, multiplicative subset $S\subset R$,
and an $R$\+module $A$, the following conditions are equivalent:
\begin{enumerate}
\renewcommand{\theenumi}{\roman{enumi}}
\item the $R$\+module $\Lambda_S(A)$ is ($S$\+separated and)
$S$\+complete;
\item the $S$\+topology on $\Lambda_S(A)$ coincides with
the projective limit topology;
\item for every element $s\in S$, the submodule $U_s\subset\Lambda_S(A)$
coincides with $s\Lambda_S(A)$;
\item for every element $s\in S$, the map $\bar\lambda_{S,A}\:
A/sA\rarrow\Lambda_S(A)/s\Lambda_S(A)$ is an isomorphism;
\item for every module $D\in R\modl_{S\secmp}$ and every $R$\+module
morphism $A\rarrow D$ there exists a unique $R$\+module morphism
$\Lambda_S(A)\rarrow D$ making the triangle diagram
$A\rarrow\Lambda_S(A)\rarrow D$ commutative;
\item the $R$\+module\/ $\coker(\lambda_{S,A})$ is $S$\+divisible;
\item $\Lambda_S(A)$ is an $S$\+pure submodule in\/~$\Pi_A$.
\end{enumerate}
\end{thm}

\begin{proof}
 The condition~(ii) implies~(i), because a projective limit is
always complete in its projective limit topology.
 The implication (iii)~$\Longrightarrow$~(ii) is clear (see
the discussion in the proof of
Proposition~\ref{countable-completion-complete}(a)).

 To compare~(iii) with~(iv), notice that we have in fact two
natural maps
\begin{equation} \label{yekutieli-direct-sum-quotients}
 A/sA\lrarrow\Lambda_S(A)/s\Lambda_S(A)\lrarrow A/sA
\end{equation}
with the composition equal to the identity map.
 The condition~(iv) says that the leftmost arrow is an isomorphism
(or surjective), while (iii)~means that the rightmost arrow is
an isomorphism (or injective).
 These are clearly two equivalent conditions.

 To see that~(i) implies (iii) and~(iv), we notice that
\eqref{yekutieli-direct-sum-quotients} implies that $\Lambda_S(A)$
is naturally a direct summand in $\Lambda_S(\Lambda_S(A))$ (this
observation comes from~\cite[Theorem~1.5 and Corollary~1.7]{Yek1}).
 When the complementary direct summand vanishes, one easily
concludes that the maps~\eqref{yekutieli-direct-sum-quotients}
are isomorphisms.
 To make this argument more explicit, one can follow~\cite{Mat} in
considering the two morphisms
$$
 \lambda_{S,\Lambda_S(A)}\ \text{and}\ \Lambda_S(\lambda_{S,A})\:
 \Lambda_S(A)\lrarrow\Lambda_S(\Lambda_S(A)).
$$
 Assume that $\Lambda_S(A)$ is $S$\+complete, that is the map
$\lambda_{S,\Lambda_S(A)}$ is surjective.
 Let $x$ be an element in $\Lambda_S(A)$; then there exists
$y\in\Lambda_S(A)$ such that $\lambda_{S,\Lambda_S(A)}(y)=
\Lambda_S(\lambda_{S,A})(x)$.
 Let the components of~$x$ be $(x_t+tA)_{t\in S}$ and the components
of~$y$ be $(y_t+tA)_{t\in S}$.
 Then we have
$$
 \lambda_{S,\Lambda_S(A)}(y)=(y+t\Lambda_S(A))_{t\in S}
 \quad\text{and}\quad
 \Lambda_S(\lambda_{S,A})(x)=
 (\lambda_{S,A}(x_t)+t\Lambda_S(A))_{t\in S}.
$$
 Hence $y-\lambda_{S,A}(x_t)\in t\Lambda_S(A)$.
 Comparing the $t$\+components, we obtain $y_t-x_t\in tA$, hence
$y=x$ in $\Lambda_S(A)$.
 We have shown that
$$
 x-\lambda_{S,A}(x_s)\in s\Lambda_S(A) \quad
 \text{for all $s\in S$},
$$
and it follows that $U_s=s\Lambda_S(A)$.

 Applying~(v) in the case of an $R$\+module $D$ annihilated by
the action of~$s$, one can see that (v) implies~(iv).
 Computing $\Hom_R(A,D)$ as $\varprojlim_{s\in S}\Hom_R(A,D/sD)$
and similarly for $\Hom_R(\Lambda_S(A),D)$, one concludes that
(iv) implies~(v) (see the last paragraph of the proof of
Proposition~\ref{countable-completion-complete}(b)).

 To prove that (vi) implies~(iii), suppose
that $x$~is an element of~$U_s$.
 Using~(vi), we can present it in the form $x=\lambda_{S,A}(a)+sy$
with $a\in A$ and $y\in\Lambda_S(A)$.
 Now the $s$\+component of $\lambda_{S,A}(a)$ vanishes in
$A/sA$, because the $s$\+components of~$x$ and~$sy$ vanish.
 Hence $a\in sA$ and $x\in s\Lambda_S(A)$.
 Conversely, let $x$ be an element in $\Lambda_S(A)$ with
the components $(x_t+tA)_{t\in S}$.
 Set $y=x-\lambda_{S,A}(x_s)$.
 Then $y\in U_s$, and by~(iii) we can conclude that
$y\in s\Lambda_S(A)$.
 So $x\in\lambda_{S,A}(x_s)+s\Lambda_S(A)$, providing~(vi).

 To check that (vii) is equivalent to~(iii), let us show that
$U_s=\Lambda_S(A)\cap s\Pi_A\subset\Pi_A$ for every $R$\+module
$A$ and every $s\in S$.
 Indeed, the inclusion $\Lambda_S(A)\cap s\Pi_A\subset U_s$
is clear.
 Conversely, let $x$ be an element of $U_s$ with the components
$(x_t+tA)_{t\in S}$.
 Then $x_s\in sA$.
 By the definition of the projective limit, we have
$x_{st}-x_t\in tA$ and $x_{st}-x_s\in sA$.
 Replacing $x_t$ by $x'_t=x_{st}$ fot every $t\in S$, we get
$x=(x'_t+tA)_{t\in S}$ and $x'_t\in sA$ for every $t\in S$.
 Thus $x\in s\Pi_A$.
\end{proof}

 The following lemma is our (weak) version of~\cite[Lemma~6.9]{Mat}.
(See also~\cite[Lemmas~1.6\+-1.7]{GT}.)

\begin{lem} \label{torsion-free-completion-complete}
 For every $S$\+torsion-free $R$\+module $A$, the equivalent
conditions of Theorem~\ref{when-completion-complete} hold.
\end{lem}

\begin{proof}
 Let us check the condition~(vii).
 Let $x$ be an element in $\Lambda_S(A)\cap s\Pi_A$.
 Then $x=(x_t+tA)_{t\in S}$ with $x_t=sz_t$ for every $t\in S$.
 Set $y_t=z_{st}$.
 Then for every $r\in S$ with $r\ge t$ we have $s(y_r-y_t)=
sz_{sr}-sz_{st}=x_{sr}-x_{st}\in stA$, because $sr\ge st$.
 Since there is no $s$\+torsion in $A$, we can conclude that
$y_r-y_t\in tA$.
 So there is an element $y\in\Lambda_S(A)$ with the components
$(y_t+tA)_{t\in S}$.
 Finally, $sy_t=sz_{st}=x_{st}\in x_t+tA$ for every $t\in S$,
hence $x=sy$ in $\Lambda_S(A)$.
\end{proof}

 The following theorem is our version of~\cite[Theorem~6.10]{Mat}.

\begin{thm} \label{delta-lambda-thm}
 Let $B$ be an $R$\+module with bounded $S$\+torsion.
 Then\par
\textup{(a)} the $R$\+module $\Lambda_S(B)$ is $S$\+complete; \par
\textup{(b)} the kernel and cokernel of the morphism
$\lambda_{S,B}\:B\rarrow\Lambda_S(B)$ are $(S^{-1}R)$\+mod\-ules; \par
\textup{(c)} the natural morphism $\beta_{S,B}\:\Delta_S(B)\rarrow
\Lambda_S(B)$ is an isomorphism; \par
\textup{(d)} the morphism $\lambda_{S,B}$ induces an isomorphism\/
$\Gamma_S(B)\simeq\Gamma_S(\Lambda_S(B))$; \par
\textup{(e)} there is a short exact sequence
$$
 0\lrarrow\Gamma_S(B)\lrarrow\Lambda_S(B)\lrarrow
 \Lambda_S(B/\Gamma_S(B))\lrarrow0.
$$
\end{thm}

\begin{proof}
 Set $A=B/\Gamma_S(B)$; and let us first prove the assertions of
the theorem for an $S$\+torsion-free $R$\+module~$A$
(``Case~I'' in~\cite{Mat}).
 For the $R$\+module $A$ in place of $B$, the assertion~(a) holds by
Lemma~\ref{torsion-free-completion-complete}; and from
Theorem~\ref{when-completion-complete}(vi) we also see that
$\coker(\lambda_{S,A})$ is an $S$\+divisible $R$\+module.
 Furthermore, $\ker(\lambda_{S,A})=\bigcap_{s\in R}sA$ is
$S$\+torsion-free (as a submodule in $A$) and $S$\+divisible,
because $a\in\ker(\lambda_{S,A})$ implies that for every $s$, $t\in S$
there exist $b$, $c\in A$ with $a=sb=stc$.
 Since $A$ is $s$\+torsion-free, one has $b=tc$, that is
$b\in\ker(\lambda_{S,A})$.

 To check that $\coker(\lambda_{S,A})$ is $S$\+torsion-free,
consider an element $x=(x_t+tA)\in \Lambda_S(A)$ such that
$sx=\lambda_{S,A}(a)$ for some $a\in A$.
 Then $a\in sA$, so there is $b\in A$ such that $a=sb$.
 For every $t\in S$ we have $sx_t-sb\in tA$, hence
$sx_t\equiv sx_{st}\equiv sb\bmod stA$.
 Since $A$ has no $s$\+torsion, it follows that $x_t\equiv b
\bmod tA$ and $x=\lambda_{S,A}(b)$.
 This proves part~(b) for the $R$\+module $A$, as we recall that
an $R$\+module is $S$\+torsion-free and $S$\+divisible if and only
if it is an $(S^{-1}R)$\+module.

 In addition, the $R$\+module $A/\ker(\lambda_{S,A})$ is
$S$\+torsion-free, because $sa\in\ker(\lambda_{S,A})$ for $s\in S$
and $a\in A$ implies that there is $b\in\ker(\lambda_{S,A})$
such that $sb=sa$, according to the above.
 Hence $b=a$ and $a\in\ker(\lambda_{S,A})$.
 The $R$\+modules $\im(\lambda_{S,A})$ and $\coker(\lambda_{S,A})$
being $S$\+torsion-free, it follows that the $R$\+module
$\Lambda_S(A)$ is $S$\+torsion-free, too.
 So part~(d) holds for $A$; and part~(e) is trivial in this case.

 To prove part~(c), one applies Lemma~\ref{rich-functors}.
 In view of part~(b) and Lemma~\ref{delta-produces-contramodule}(b),
there exists a unique $R$\+module morphism $\Lambda_S(A)\rarrow
\Delta_S(A)$ making the triangle diagram $A\rarrow\Lambda_S(A)
\rarrow\Delta_S(A)$ commutative.
 Using the uniqueness assertion of Lemma~\ref{rich-functors},
one shows that the compositions $\Delta_S(A)\rarrow\Lambda_S(A)
\rarrow\Delta_S(A)$ and $\Lambda_S(A)\rarrow\Delta_S(A)\rarrow
\Lambda_S(A)$ are the identity maps.
 This finishes the proof of the theorem in the case of
an $S$\+torsion-free $R$\+module~$A$.

 Returning to the general case of an $R$\+module $B$ with bounded
$S$\+torsion (``Case~II'' in~\cite{Mat}), consider the short exact
sequence of $R$\+modules
 $$
 0\lrarrow \Gamma_S(B)\lrarrow B\lrarrow A\lrarrow0.
$$
 Applying the functors $\Delta_S$ and $\Lambda_S$ and the natural
transformation~$\beta_S$, we obtain a commutative diagram
$$
\begin{diagram}
\node{0}\arrow{e}\node{\Delta_S(\Gamma_S(B))}\arrow{s}\arrow{e}
\node{\Delta_S(B)}\arrow{s}\arrow{e}\node{\Delta_S(A)}
\arrow{s}\arrow{e}\node{0} \\
\node{0}\arrow{e}\node{\Lambda_S(\Gamma_S(B))}\arrow{e}
\node{\Lambda_S(B)}\arrow{e}\node{\Lambda_S(A)}
\end{diagram}
$$
 The upper line is exact, since $\Ext_R^0(K^\bu,A)=0$
(by Lemma~\ref{tor-with-k}(c), as $S^{-1}R/R$ is $S$\+torsion and
$A$ is $S$\+torsion-free) and $\Ext_R^2(K^\bu,\Gamma_S(B))=0$ (by
Lemma~\ref{about-contramodules-lemma}(b)).
 The lower line is exact (at the leftmost and the middle terms)
because $\Gamma_S(B)$ is an $S$\+pure submodule in $B$ and
the functor of projective limit is left exact.

 Furthermore, the morphism $\Delta_S(A)\rarrow\Lambda_S(A)$ is
an isomorphism, as we have already proven in part~(c) of
``Case~I'' above.
 It follows that the morphism $\Lambda_S(B)\rarrow\Lambda_S(A)$
is surjective and the lower line of the diagram is a short exact
sequence, too.
 Finally, the morphism $\Delta_S(\Gamma_S(B))\rarrow
\Lambda_S(\Gamma_S(B))$ is an isomorphism, because the morphism
$\Gamma_S(B)\rarrow\Delta_S(\Gamma_S(B))$ is an isomorphism (by
Lemma~\ref{about-contramodules-lemma}(a\+b)) and
the morphism $\Gamma_S(B)\rarrow\Lambda_S(\Gamma_S(B))$ is
(as it is clear from the construction of the $S$\+completion
functor~$\Lambda_S$).
 We have proved parts~(c) and~(e) for the $R$\+module~$B$.
 As the $R$\+module $\Lambda_S(B/\Gamma_S(B))=\Lambda_S(A)$ is
$S$\+torsion-free, part~(d) follows from~(e).

 Very little remains to be done.
 Since $\lambda_{S,\.\Gamma_S(B)}$ is an isomorphism, the kernel and
cokernel of the morphism $\lambda_{S,B}$ are isomorphic to,
respectively, the kernel and cokernel of the morphism
$\lambda_{S,A}$.
 This proves part~(b).
 Now we also know that $\coker(\lambda_{S,B})\simeq
\coker(\lambda_{S,A})$ is an $S$\+divisible $R$\+module.
 Applying Theorem~\ref{when-completion-complete}\,(vi)%
$\Longrightarrow$(i), we deduce the assertion~(a) for
the $R$\+module~$B$.
\end{proof}

 From our point of view, the functor $\Delta_S=\Ext^1_R(K^\bu,{-})$
is of primary importance, while the significance of the $S$\+completion
functor $\Lambda_S$ lies in the fact that, according to
Theorem~\ref{delta-lambda-thm}(c), it sometimes allows to compute
the functor~$\Delta_S$.

\begin{cor} \label{bounded-torsion-sepcomplete-contramodule}
 An $R$\+module $B$ with bounded $S$\+torsion is an $S$\+contramodule
if and only if it is $S$\+separated and $S$\+complete.
\end{cor}

\begin{proof}
 Any $S$\+separated $S$\+complete $R$\+module is an $S$\+contramodule
by Lemma~\ref{completion-contramodule}(a).
 Conversely, for an $R$\+module $B$ with bounded $S$\+torsion
the natural map $\beta_{S,B}\:\Delta_S(B)\allowbreak
\rarrow\Lambda_S(B)$ is an isomorphism by
Theorem~\ref{delta-lambda-thm}(c).
 Suppose that $B$ is an $S$\+contra\-module; then, by
Lemma~\ref{about-contramodules-lemma}(a), the map
$\delta_{S,B}\:B\rarrow\Delta_S(B)$ is an isomorphism.
 Hence the map $\lambda_{S,B}\:B\rarrow\Lambda_S(B)$ is
an isomorphism.
 By Proposition~\ref{countable-completion-complete}(a) and
Theorem~\ref{delta-lambda-thm}(a), the $R$\+module $\Lambda_S(B)$
is $S$\+separated and $S$\+complete.
 Thus $B$ is $S$\+separated and $S$\+complete. \hfuzz=3.5pt
\end{proof}

 The following corollary is to be compared with~\cite[Lemma~2.5]{Pmgm}.

\begin{cor} \label{flat-bounded-torsion}
 Assume that the $R$\+module $R$ has bounded $S$\+torsion.
 Then for every flat $R$\+module $F$ the conclusions of
Theorem~\ref{delta-lambda-thm} hold.
 In particular, the morphism $\beta_{S,F}\:\Delta_S(F)\rarrow
\Lambda_S(F)$ is an isomorphism.
\end{cor}

\begin{proof}
 Using the Govorov--Lazard description of flat $R$\+modules as
filtered inductive limits of (finitely generated) projective
$R$\+modules, one easily shows that if $r\Gamma_S(R)=\nobreak0$ for
some $r\in S$ then $r\Gamma_S(F)=0$ for every flat $R$\+module~$F$.
\end{proof}

\Section{Projective $S$-contramodule $R$-modules}

 Let $S$ be a multiplicative subset in a commutative ring~$R$.
 For every injective $R$\+module $J$, the $S$\+torsion $R$\+module
$E=\Gamma_S(J)$ is an injective object of the abelian category of
$S$\+torsion $R$\+modules $R\modl_{S\tors}$.
 There are enough injective objects of this form in $R\modl_{S\tors}$,
so an $S$\+torsion $R$\+module is an injective object in
$R\modl_{S\tors}$ if and only if it is a direct summand of
an $R$\+module of the form $\Gamma_S(J)$, where $J$ is an injective
$R$\+module.

 Our aim in this section is to describe the projective objects in
the category of $S$\+contramodule $R$\+modules, under suitable
assumptions.
 We also discuss complexes of $R$\+modules with $S$\+torsion or
$S$\+contramodule cohomology modules generally.
 The following proposition, which is our version
of~\cite[Proposition~5.1]{Mat} and~\cite[Proposition~4.1]{FS0},
holds for any multiplicative subset $S$ in a commutative ring~$R$.

\begin{prop} \label{endomorphism-ring-commutative}
 The ring\/ $\fR=\Hom_{\sD^\b(R\modl)}(K^\bu,K^\bu)$ is commutative.
\end{prop}

\begin{proof}
 The point is that the complex $K^\bu[-1]$ is a unit object in a certain
tensor (monoidal) triangulated category structure.
 One can consider the unbounded derived category of $R$\+modules
$\sD(R\modl)$ with the tensor product functor $\ot_R^\boL$ defined in
terms of homotopy flat $R$\+module resolutions; or one can restrict
oneself to the bounded above derived category $\sD^-(R\modl)$, where
the tensor product can be defined in terms of the conventional
resolutions by complexes of flat $R$\+modules.
 One can even consider the derived category $\sD^\b(R\modl^\fl)\subset
\sD^\b(R\modl)$ of bounded complexes of flat $R$\+modules, with
the obvious tensor product structure on it.

 In any event, the full subcategory of complexes with $S$\+torsion
cohomology modules $\sD_{S\tors}(R\modl)\subset\sD(R\modl)$ is a tensor
ideal in $\sD(R\modl)$, and similarly in the bounded situations.
 Indeed, a complex of $R$\+modules $M^\bu\in\sD(R\modl)$ belongs
to $\sD_{S\tors}(R\modl)$ if and only if $S^{-1}R\ot_RM^\bu$ is
a zero object in $\sD(R\modl)$; and $S^{-1}R\ot_R(M^\bu\ot_R^\boL N^\bu)
\simeq (S^{-1}R\ot_RM^\bu)\ot_R^\boL N^\bu$ for every
$M^\bu$, $N^\bu\in\sD(R\modl)$.
 Furthermore, the tensor ideal $\sD_{S\tors}(R\modl)$ is itself
a tensor category, and as such it has its own unit object, which
is $K^\bu[-1]\in\sD_{S\tors}(R\modl)$.

 Now, the endomorphism semigroup of the unit object of any monoidal
category is commutative.
 Specifically in our case, it is important that there is an isomorphism
\begin{equation} \label{unit-object}
 K^\bu[-1]\ot_R^\boL K^\bu[-1]=K^\bu[-1]\ot_R K^\bu[-1]\simeq K^\bu[-1]
\end{equation}
in $\sD(R\modl)$, and this isomorphism is compatible with
the endomorphisms of $K^\bu[-1]$, that is for every endomorphism
$f\:K^\bu[-1]\rarrow K^\bu[-1]$ in $\sD(R\modl)$ the endomorphisms
$\id\ot f$ and $f\ot\id$ of the object $K^\bu[-1]\ot_R K^\bu[-1]$ are
identified with~$f$ by the natural isomorphism~\eqref{unit-object}.
 Now we have
\begin{multline*}
 fg=fg\ot\id=(f\ot\id)(g\ot\id)=(f\ot\id)(\id\ot g)=f\ot g \\
 =(\id\ot g)(f\ot\id)=(g\ot\id)(f\ot\id)=gf\ot\id=gf
\end{multline*}
for any pair of endomorphisms $f$, $g\:K^\bu[-1]\rarrow K^\bu[-1]$
in $\sD(R\modl)$.
 In fact, this argument proves that the whole graded ring of
endomorphisms
\begin{equation} \label{graded-ring-of-endomorphisms}
 \Hom_{\sD^\b(R\modl)}(K^\bu,K^\bu[*])\simeq
 \Hom_{\sD^\b(R\modl)}(K^\bu[-1],\>K^\bu[-1+*])
\end{equation}
of the object $K^\bu$ or $K^\bu[-1]$ in the triangulated category
$\sD^\b(R\modl)$ is supercommutative.
\end{proof}

 Clearly, there is a natural ring homomorphism $R\rarrow\fR$,
making $\fR$ a commutative $R$\+algebra.
 Notice also the natural $R$\+module isomorphism
$$
 \fR=\Hom_{\sD^\b(R\modl)}(K^\bu,K^\bu)\simeq\Hom_{\sD^\b(R\modl)}(K^\bu,R[1])
 =\Delta_S(R),
$$
which holds because $\Ext^*_R(K^\bu,S^{-1}R)=0$.
 On the other hand, the $S$\+completion $\Lambda_S(R)$ of
the ring $R$ is also a commutative ring and an $R$\+algebra,
since it is constructed as the projective limit of a projective
system of rings $R/sR$, \ $s\in S$ and ring homomorphisms between them.

\begin{prop}
\textup{(a)} The natural map $\beta_{S,R}\:\fR\rarrow\Lambda_S(R)$ is
a ring homomorphism. \par
\textup{(b)} Denoting by $I\subset\fR$ the kernel ideal of
the ring homomorphism $\beta_{S,R}$, one has $I^2=0$ in\/~$\fR$. \par
\textup{(c)} When the $S$\+torsion in the ring $R$ is bounded,
the map $\beta_{S,R}$ is an isomorphism.
\end{prop}

\begin{proof}
 Part~(a): by construction, $\Lambda_S(R)$ is a subring in the ring
$\Pi_R=\prod_{s\in S}R/sR$; so it suffices to check that the maps
$\fR\rarrow R/sR$ are ring homomorphisms.
 The latter maps are uniquely defined by the condition of commutativity
of the triangle diagrams $R\rarrow\fR\rarrow R/sR$ (see
Lemmas~\ref{delta-mod-s} and~\ref{completion-contramodule}(b)).
 Finally, there is a functor of reduction modulo the ideal
$sR\subset R$, acting, e.~g., between the bounded derived categories
of flat modules $\sD^\b(R\modl^\fl)\rarrow\sD^\b((R/sR)\modl^\fl)$,
or even between the unbounded derived categories $\sD(R\modl)\rarrow
\sD((R/sR)\modl)$, etc., and taking the complex $K^\bu[-1]$ to $R/sR$.
 This functor induces a map $\fR\rarrow R/sR=\Hom_{\sD^\b((R/sR)\modl)}
(R/sR,R/sR)$, which is a ring homomorphism forming a commutative
triangle diagram with the maps $R\rarrow\fR$ and $R\rarrow R/sR$.

 Part~(b): for every $S$\+torsion $R$\+module $M$, the isomorphism
$M\simeq K^\bu[-1]\ot_R M$ in the derived category $\sD^\b(R\modl)$
(as discussed in the proof of
Proposition~\ref{endomorphism-ring-commutative}) endows $M$ with
a natural $\fR$\+module structure.
 Specifically, an endomorphism $f\:K^\bu[-1]\rarrow K^\bu[-1]$ induces
the endomorphism $f\ot\id$ of the object $K^\bu[-1]\ot_R M$ in
$\sD^\b(R\modl)$.
 Representing $M$ as the inductive limit of its bounded torsion
submodules and using Lemma~\ref{delta-mod-s}, one can see that this
$\fR$\+module structure on $M$ comes from the (obvious)
$\Lambda_S(R)$\+module structure on $M$ via the ring
homomorphism~$\beta_{S,R}$.

 Furthermore, the isomorphism $M^\bu\simeq K^\bu[-1]\ot_R M^\bu$
in $\sD(R\modl)$ holds for every complex of $R$\+modules $M^\bu$
with $S$\+torsion cohomology modules.
 This isomorphism provides an action of the ring $\fR$, and in
fact even of the graded ring~\eqref{graded-ring-of-endomorphisms},
by (graded) endomorphisms of the object $M^\bu\in\sD_{S\tors}(R\modl)$.
 This action already does \emph{not} necessarily factorize through
an action of the ring~$\Lambda_S(R)$.
 In fact, the action of the ring $\fR$ by endomorphisms of
the object $K^\bu[-1]$ constructed in this way, i.~e., in terms of
the isomorphism $K^\bu[-1]\simeq K^\bu[-1]\ot_R K^\bu[-1]$ in
$\sD(R\modl)$, coincides with the action of $\fR$ by endomorphisms
of $K^\bu[-1]$ coming from the definion of $\fR$ as the endomorphism
ring of the object $K^\bu[-1]$, because we have $f\ot\id=f$ for
every morphism $f\:K^\bu[-1]\rarrow K^\bu[-1]$ in $\sD^\b(R\modl)$.

 By the definition, the natural action of $\fR$ by endomorphisms of
objects of the category $\sD_{S\tors}(R\modl)$ commutes with all
the morphisms $M^\bu\rarrow N^\bu$ between the objects of
$\sD_{S\tors}(R\modl)$.
 Set $L^\bu=K^\bu[-1]$.
 We have shown that for every morphism $f\:L^\bu\rarrow L^\bu$ in
$\sD^\b(R\modl)$ belonging to the ideal $I\subset\fR$ the induced
cohomology module endomorphisms $H^*(f)\:H^*(L^\bu)\rarrow H^*(L^\bu)$
vanish.

 Now the point is that $L^\bu$ is a two-term complex whose only
possibly nontrivial cohomology modules are $H^0(L^\bu)$ and
$H^1(L^\bu)$.
 Considering the distingushed triangle
$$
 H^0(L^\bu)\lrarrow L^\bu\lrarrow H^1(L^\bu)[-1]
 \lrarrow H^0(L^\bu)[1]
$$
in $\sD^\b(R\modl)$, one easily comes to the conclusion that
the morphism~$f$ is the composition of the morphism $L^\bu\rarrow
H^1(L^\bu)[-1]$ on one side, the morphism $H^0(L^\bu)\rarrow L^\bu$
on the other side, and an $\Ext^1_R$\+extension class
$$
 H^1(L^\bu)[-1]\lrarrow H^0(L^\bu)
$$
in the middle.
 The composition $H^0(L^\bu)\rarrow L^\bu\rarrow H^1(L^\bu)[-1]$ being
a zero morphism in $\sD^\b(R\modl)$, it follows immediately that
$f^2=0$ in~$\fR$.

 Part~(c) is a particular case of Theorem~\ref{delta-lambda-thm}(c)
or Corollary~\ref{flat-bounded-torsion}.
\end{proof}

 The following lemma is our version of~\cite[first assertion of
Theorem~2.1]{Mat} or~\cite[Theorem~1.6]{Mat2}.

\begin{lem} \label{hom-contramodule}
 The $R$\+module $\Hom_R(M,C)$ is an $S$\+contramodule whenever
either $M$ is an $S$\+torsion $R$\+module or $C$ is
an $S$\+contramodule $R$\+module.
\end{lem}

\begin{proof}
 The case when $M$ is an $S$\+torsion $R$\+module can be dealt with
similarly to the proof of Lemma~\ref{delta-produces-contramodule}(a),
with the complex $K^\bu$ replaced by a module~$M$.
 When $C$ is an $S$\+contramodule, one can consider the same spectral
sequence 
$$
 E_2^{pq}=\Ext^p_R(S^{-1}R,\Ext^q_R(M,C))\Longrightarrow
 E_\infty^{pq}=\mathrm{gr}^p\Ext_R^{p+q}(S^{-1}R\ot_RM,\>C),
$$
where $\Ext_R^n(S^{-1}M,C)=0$ for $n\le1$ by
Lemma~\ref{ext-from-s-minus-1-r-module}, hence
$E_\infty^{pq}=0$ for $p+q\le1$.
 Now $E_\infty^{0,0}=0=E_\infty^{1,0}$ implies $E_2^{0,0}=0=E_2^{1,0}$,
since no nontrivial differentials go through $E_r^{0,0}$ or $E_r^{1,0}$.
 Alternatively, the argument from~\cite[Theorem~2.1]{Mat} also works.
\end{proof}

 \emph{From this point on and for the rest of this paper we assume
that\/ $\pd_RS^{-1}R\le1$.}
 In this assumption, the classes of $S$\+cotorsion $R$\+modules
and strongly $S$\+cotorsion $R$\+modules coincide, as do the classes
of $S$\+contramodule $R$\+modules and strong $S$\+contramodule
$R$\+modules.
 We denote the full subcategory of $S$\+contramodule $R$\+modules
by $R\modl_{S\ctra}\subset R\modl$.

\begin{thm} \label{contramodule-category-thm}
\textup{(a)} The full subcategory $R\modl_{S\ctra}$ is closed under
the kernels, cokernels, extensions, and infinite products in $R\modl$.
 Therefore, the category $R\modl_{S\ctra}$ is abelian and its
embedding functor $R\modl_{S\ctra}\rarrow R\modl$ is exact. \par
\textup{(b)} For every $R$\+module $C$, the $R$\+module
$\Delta_S(C)=\Ext^1_R(K^\bu,C)$ is an $S$\+contra\-module.
 The functor $\Delta_S\:R\modl\rarrow R\modl_{S\ctra}$ is left adjoint
to the fully faithful embedding functor $R\modl_{S\ctra}\rarrow R\modl$.
\end{thm}

\begin{proof}
 Part~(a) is~\cite[Proposition~1.1]{GL} or~\cite[Theorem~1.2(a)]{Pcta}
applied to the $R$\+module $S^{-1}R$.
 The first assertion of part~(b) is
Lemma~\ref{delta-produces-contramodule}(c).
 The second assertion of part~(b) means that for every $R$\+module $C$,
every $S$\+contramodule $R$\+module $D$, and an $R$\+module morphism
$C\rarrow D$ there exists a unique $R$\+module morphism
$\Delta_S(C)\rarrow D$ making the triangle diagram
$C\rarrow\Delta_S(C)\rarrow D$ commutative.
 This follows from Lemma~\ref{rich-functors} together with the short
exact sequence~(III).
\end{proof}

 Using homotopy projective and homotopy injective $R$\+module
resolutions, one can endow the monoidal triangulated category
$\sD(R\modl)$ with a closed monoidal structure provided by
the functor
$$
 \boR\Hom_R\:\sD(R\modl)^\sop\times\sD(R\modl)\lrarrow\sD(R\modl).
$$
 Restricting oneself to bounded above complexes $L^\bu$ in
the first argument, one construct the complex $\boR\Hom_R(L^\bu,M^\bu)$
using a conventional resolution of $L^\bu$ by a (bounded above)
complex of projective $R$\+modules.
 One can even ask $L^\bu$ to belong to the homotopy category of bounded
complexes of projective $R$\+modules $\Hot^\b(R\modl^\proj)\subset
\sD^\b(R\modl)$.
 Using any of these points of view, one notices the natural isomorphism
$\boR\Hom_R(K^\bu[-1],C)\simeq C$ in $\sD^\b(R\modl)$ for every
$S$\+contramodule $R$\+module~$C$.
 This isomorphism endows $C$ with a natural structure of $\fR$\+module
(cf.~\cite[Theorem~2.7]{Mat2}, where it is also explained
how to show that such an $\fR$\+module structure on $C$ is unique).

 Similarly one can construct an action of the ring $\fR$, and even
of the graded ring~\eqref{graded-ring-of-endomorphisms}, in every
object of the full subcategory $\sD_{S\ctra}(R\modl)\subset\sD(R\modl)$
of complexes with $S$\+contramodule cohomology modules in $\sD(R\modl)$.
 The following proposition provides some details.

\begin{prop}
\textup{(a)} A complex of $R$\+modules $A^\bu\in\sD(R\modl)$ has
$S$\+contramodule cohomology modules if and only if\/
$\boR\Hom_R(S^{-1}R,A^\bu)=0$.
 Hence one has $A^\bu\simeq\boR\Hom_R(K^\bu[-1],A^\bu)$ in\/
$\sD(R\modl)$ for every $A^\bu\in\sD_{S\ctra}(R\modl)$. \par
\textup{(b)} The complex of $R$\+modules\/ $\boR\Hom_R(M^\bu,A^\bu)$
has $S$\+contramodule cohomology modules whenever either a complex
$M^\bu$ has $S$\+torsion cohomology modules, or a complex $A^\bu$
has $S$\+contramodule cohomology modules.
\end{prop}

\begin{proof}
 For every complex of $R$\+modules $A^\bu$, there are natural short
exact sequences
\begin{multline*}
 0\lrarrow\Ext^1_R(S^{-1}R,H^{n-1}(A^\bu))\lrarrow
 H^n(\boR\Hom_R(S^{-1}R,A^\bu))\\ \lrarrow\Hom_R(S^{-1}R,H^n(A^\bu))
 \lrarrow0
\end{multline*}
for all $n\in\Z$.
 This proves part~(a); and part~(b) follows from the isomorphisms
\begin{multline*}
 \boR\Hom_R(S^{-1}R,\.\boR\Hom_R(M^\bu,A^\bu))\simeq
 \boR\Hom_R(S^{-1}R\ot_RM^\bu,\>A^\bu) \\ \simeq
 \boR\Hom_R(M^\bu,\.\boR\Hom_R(S^{-1}R,A^\bu)).
\end{multline*}
\end{proof}

 For any set $X$ and an $R$\+module $M$, we denote by $M[X]$
the direct sum of $X$ copies of the $R$\+module~$M$.
 We set $\fR[[X]]=\Delta_S(R[X])$.
 The $R$\+module $\fR[[X]]$ is called the \emph{free $S$\+contramodule
$R$\+module} generated by the set~$X$.
 According to Theorem~\ref{contramodule-category-thm}(b), we have
$$
 \Hom_R(\fR[[X]],D)\simeq\Hom_R(R[X],D)\simeq D^X
$$
for every $S$\+contramodule $R$\+module~$D$.
 It follows that free $S$\+contramodule $R$\+modules are projective
objects of the category $R\modl_{S\ctra}$.
 Furthermore, there are enough of them, as every $S$\+contramodule
$R$\+module $D$ is a quotient module of the free $S$\+contramodule
$R$\+module $\fR[[X]]$ with the set of generators $X=D$.
 It follows that an $S$\+contramodule $R$\+module is a projective
object of $R\modl_{S\ctra}$ if and only if it is a direct summand
of a free $S$\+contramodule $R$\+module.

 According to Corollary~\ref{flat-bounded-torsion}, when
the $S$\+torsion in the ring $R$ is bounded, one has
$$
 \fR[[X]]=\Delta_S(R[X])\simeq\Lambda_S(R[X])=
 \varprojlim\nolimits_{s\in S} R/sR[X].
$$

\Section{The Triangulated Equivalence} \label{triangulated-equivalence}

 Let $S$ be a multiplicative subset in a commutative ring~$R$.
 Since the middle of the previous section, we keep assuming that
the projective dimension of the $R$\+module $S^{-1}R$ does not
exceed~$1$.

 Following the notation in~\cite[Section~1]{Pmgm}
and~\cite[Section~7]{Pcta}, for every $R$\+module
$M$ we denote by $\check C^\bu(M)\sptilde$ the two-term complex
$M\rarrow S^{-1}M$, with the term $M$ sitting in cohomological
degree~$0$ and the term $S^{-1}M$ sitting in degree~$1$.
 In other words, $\check C^\bu(M)\sptilde=K^\bu[-1]\ot_RM$.
 Given a complex of $R$\+modules $M^\bu$, the notation
$\check C^\bu(M^\bu)\sptilde$ stands for the total complex of
the corresponding bicomplex with two rows.

 The next two lemmas are almost obvious.

\begin{lem} \label{complex-check-c-tilde}
\textup{(a)} For every $R$\+module $M$, the cohomology modules
$H^*\check C^\bu(M)\sptilde$ of the complex $\check C^\bu(M)\sptilde$
are $S$\+torsion $R$\+modules. \par
\textup{(b)} For every $R$\+module $M$, one has an isomorphism
$H^0\check C^\bu(M)\sptilde\simeq\Gamma_S(M)$. \par
\textup{(c)} For every $S$\+torsion $R$\+module $M$, the natural
projection $\check C^\bu(M)\sptilde\rarrow M$ is an isomorphism
of complexes.
\end{lem}

\begin{proof}
 Part~(a) holds, because the complex $S^{-1}R\ot_R\check C^\bu(M)
\sptilde$ is contractible.
 Part~(b) is Lemma~\ref{tor-with-k}(d).
 Part~(c) holds, because $S^{-1}M=0$ when $M$ is $S$\+torsion.
\end{proof}

 Following the notation in~\cite[Section~2]{Pmgm}
and~\cite[Section~7]{Pcta}, we choose a left projective resolution
$\overline T^\bu=(\overline T^0\to \overline T^1)$ of
the $R$\+module $S^{-1}R$ (notice the unusual indexing/cohomological
grading: it is presumed that $H^0(\overline T^\bu)=0$ and
$H^1(\overline T^\bu)=S^{-1}R$).
 Furthermore, we lift the morphism $R\rarrow S^{-1}R$ to an $R$\+module
morphism $R\rarrow T^1$ and denote by $T^\bu=(T^0\to T^1)$
the two-term complex $(R\oplus\overline T^0\to\overline T^1)$
quasi-isomorphic to $K^\bu[-1]$.
 The aim of the grading shift is to have a natural (projection)
morphism $T^\bu\rarrow R$.
 Given an $R$\+module $C$, we will sometimes view the complex
$\Hom_R(T^\bu,C)$ as a \emph{homological} complex sitting in
the homological degrees~$0$ and $1$, which are denoted by
the lower indices.

\begin{lem} \label{complex-t}
\textup{(a)} For every $R$\+module $C$, the homology modules
$H_*\Hom_R(T^\bu,C)$ of the complex\/ $\Hom_R(T^\bu,C)$ are
$S$\+contramodule $R$\+modules. \par
\textup{(b)} For every $R$\+module $C$, one has
$H_0\Hom_R(T^\bu,C)\simeq\Delta_S(C)$. \par
\textup{(c)} For every $S$\+contramodule $R$\+module $C$,
the morphism of complexes $T^\bu\rarrow R$ induces a quasi-isomorphism
of complexes of $R$\+modules $C\rarrow\Hom_R(T^\bu,C)$.
\end{lem}

\begin{proof}
 Part~(a) is Lemma~\ref{delta-produces-contramodule}(a,c).
 Part~(b) is the definition of $\Delta_S=\Ext^1_R(K^\bu,{-})$.
 Part~(c) means that the complex $\Hom_R(\overline T^\bu,C)$ is acyclic,
which is the definition of an $S$\+contramodule $R$\+module.
\end{proof}

 The following proposition is essentially proved
in~\cite[Proposition~3.3]{Pmgm}.
 In the paper~\cite{Pmgm}, it was being applied to the case of
the \v Cech DG\+algebra of a finite sequence of elements
$\mathbf s$ in a commutative ring $R$ (which was denoted by
$C_{\mathbf s}^\bu(R)$ in~\cite{Pmgm} and would be denoted by
$\check C_{\mathbf s}^\bu(R)$ in the present paper's notation system).
 In this paper, we will apply this result to the $R$\+algebra
$\check C=S^{-1}R$ (viewed as a DG\+algebra concentrated in
the cohomological degree~$0$).

 Let $\check C^\bu$ be a finite complex of $R$\+modules whose
terms $\check C^n$ are flat $R$\+modules of finite projective
dimension.
 Suppose that $\check C^\bu$ is endowed with the structure of
an (associative and unital, not necessarily commutative)
DG\+algebra over the ring $R$, and that the following
condition is satisfied: the three morphisms of complexes
\begin{equation} \label{three-quasi-isomorphisms}
 \check C^\bu\.\birarrow\.\check C^\bu\ot_R\check C^\bu\rarrow
 \check C^\bu
\end{equation}
provided by the unit and multiplication in the DG\+algebra
$\check C^\bu$ are quasi-isomorphisms of complexes of
$R$\+modules.

 Let $\st$ be one of the derived category symbols~$\b$, $+$, $-$,
or~$\varnothing$.
 By the definition, the derived category $\sD^\st(\check C^\bu\modl)$
is constructed by inverting the class of quasi-isomorphisms in
the homotopy category of $\st$\+bounded left DG\+modules
over the DG\+ring~$\check C^\bu$.
 Denote by
$$
 k_*\:\sD^\st(\check C^\bu\modl)\lrarrow\sD^\st(R\modl)
$$
the functor of restriction of scalars with respect to the morphism
of DG\+rings $k\:R\rarrow\check C^\bu$.

\begin{prop}  \label{semiorthogonal-mgm}
\textup{(a)} The triangulated functor~$k_*$ has a left adjoint
functor~$k^*$ and a right adjoint functor~$\boR k^!$,
$$
k^*,\ \boR k^!\.\:\sD^\st(R\modl)\lrarrow
\sD^\st(\check C^\bu\modl). 
$$ \par
\textup{(b)} The compositions $k^*\circ k_*$ and\/ $\boR k^!\circ k_*$
are isomorphic to the identity functors on the category\/
$\sD^\st(\check C^\bu\modl)$, the functor~$k_*$ is fully faithful, and
the functors~$k^*$ and\/~$\boR k^!$ are Verdier quotient functors. \par
\textup{(c)} The passage to the quotient category by the image of
the functor~$k_*$ establishes an equivalence between the kernels of
the functors~$k^*$ and~$\boR k^!$,
$$
 \ker(k^*)\simeq\sD^\st(R\modl)/\im k_*\simeq\ker(\boR k^!).
$$
\end{prop}

 In other words, there are two semiorthogonal decompositions in
the triangulated category\/ $\sD^\st(R\modl)$, one of them formed by
the two full subcategories\/ $\im(k_*)$ and\/ $\ker(k^*)$, and
the other one by the two full subcategories\/ $\ker(\boR k^!)$ and\/
$\im(k_*)$.
 In particular, $\im(k_*)$ is a thick subcategory in $\sD^\st(R\modl)$.

\begin{proof}
 Part~(a) does not depend on the assumption that
the maps~\eqref{three-quasi-isomorphisms} are quasi-isomorphisms;
parts~(b) and~(c) do.
 The assumption of finite projective dimension of the $R$\+modules
$\check C^n$ is irrelevant in the case of the derived category
symbol $\st=+$ or~$\varnothing$, and relevant only for $\star=\b$
or~$-$ (and insofar as the functor $\boR k^!$ is concerned).
 All the assertions in part~(b) are equivalent to each other for
purely formal reasons applicable to triangulated functors generally,
and part~(c) is a purely formal restatement of part~(b).
 We refer to~\cite[proof of Proposition~3.3]{Pmgm} for the details
of the argument.
\end{proof}

 In other words, Proposition~\ref{semiorthogonal-mgm} says that
the DG\+algebra morphism $R\rarrow\check C$ gives rise to
a ``recollement'' of triangulated categories for every symbol
$\star=\b$, $+$, $-$, or~$\varnothing$.
 In the case of $\check C=S^{-1}R$, the ring homomorphism $k\:R\rarrow
S^{-1}R$ is a \emph{homological ring epimorphism} in the sense
of~\cite[Section~4]{GL}, which is also a sufficient condition for
the conclusions of the proposition to hold (for $\star=\varnothing$).
 For a generalization to arbitrary morphisms of associative
DG\+rings, see~\cite[Theorem~3.9]{Pa}.

 Generalizing our previous notation from the case $\check C=S^{-1}R$
to a DG\+algebra $\check C^\bu$ as above, denote by
$\check C^\bu{}\sptilde$ the cocone (that is, the cone shifted
by~$[-1]$) of the morphism of complexes $R\rarrow\check C^\bu$.
 Choose a finite complex of projective $R$\+modules $\overline T^\bu$
mapping quasi-isomorphically onto $\check C^\bu$, lift the morphism
$R\rarrow \check C^\bu$ to a morphism of complexes
$R\rarrow \overline T^\bu$, and denote by $T^\bu$ the cocone of
the latter morphism of complexes.
 The complexes $\check C^\bu{}\sptilde$ and $T^\bu$ are, by
construction, quasi-isomorphic to each other and endowed with
morphisms of complexes $\check C^\bu{}\sptilde\rarrow R$ and
$T^\bu\rarrow R$.

 The following proposition is essentially proved in~\cite[first
half of the proof of Theorem~3.4]{Pmgm}.

\begin{prop} \label{kernel-categories-mgm}
\textup{(a)} The functor\/ $\sD^\st(R\modl)\rarrow\ker(k^*)$ right
adjoint to the embedding\/ $\ker(k^*)\rarrow\sD^\st(R\modl)$ takes
a complex of $R$\+modules $M^\bu$ to the complex
$\check C^\bu{}\sptilde\ot_R M^\bu$. 
 The full subcategory\/ $\ker(k^*)\subset\sD^\st(R\modl)$
consists precisely of all the complexes of $R$\+modules $M^\bu$ for
which the morphism of complexes $\check C^\bu{}\sptilde\ot_RM^\bu
\rarrow M^\bu$ induced by the morphism $\check C^\bu{}\sptilde
\rarrow R$ is a quasi-isomorphism. \par
\textup{(a)} The functor\/ $\sD^\st(R\modl)\rarrow\ker(\boR k^!)$
left adjoint to the embedding\/ $\ker(\boR k^!)\allowbreak\rarrow
\sD^\st(R\modl)$ takes a complex of $R$\+modules $A^\bu$ to
the complex\/ $\Hom_R(T^\bu,A^\bu)$.
 The full subcategory\/ $\ker(\boR k^!)\subset
\sD^\st(R\modl)$ consists precisely of all the complexes of
$R$\+modules $A^\bu$ for which the morphism of complexes
$A^\bu\rarrow\Hom_R(T^\bu,A^\bu)$ induced by the morphism
$T^\bu\rarrow R$ is a quasi-isomorphism. \qed
\end{prop}

 Now we return to the situation with the $R$\+algebra
$\check C=S^{-1}R$.
 As in the previous section, we denote by $\sD_{S\tors}^\st(R\modl)
\subset\sD^\st(R\modl)$ the full triangulated subcategory in
the derived category $\sD^\st(R\modl)$ formed by all the complexes
of $R$\+modules $M^\bu$ whose cohomology modules $H^*(M^\bu)$ are
$S$\+torsion $R$\+modules.
 Similarly, $\sD_{S\ctra}^\st(R\modl)\subset\sD^\st(R\modl)$ denotes
the full triangulated subcategory in $\sD^\st(R\modl)$ formed by all
the complexes of $R$\+modules $A^\bu$ whose cohomology modules
$H^*(A^\bu)$ are $S$\+contramodule $R$\+modules.

\begin{lem} \label{kernel-categories-computed}
 Let $k$ be the $R$\+algebra homomorphism $R\rarrow S^{-1}R$.  Then \par
\textup{(a)} the full subcategory\/ $\ker(k^*)\subset\sD^\st(R\modl)$
coincides with\/ $\sD^\st_{S\tors}(R\modl)\subset\sD^\st(R\modl)$; \par
\textup{(b)} the full subcategory\/ $\ker(\boR k^!)\subset
\sD^\st(R\modl)$ coincides with\/ $\sD^\st_{S\ctra}(R\modl)\subset
\sD^\st(R\modl)$. \par
\end{lem}

\begin{proof}
 This is similar to~\cite[second half of the proof of
Theorem~3.4]{Pmgm}.
 The argument is based on Proposition~\ref{kernel-categories-mgm}
and Lemmas~\ref{complex-check-c-tilde}\+-\ref{complex-t}.

 To prove that $\ker(k^*)\subset\sD_{S\tors}^\st(R\modl)$, we have to
show that the cohomology modules of a complex of $R$\+modules $M^\bu$
are $S$\+torsion $R$\+modules whenever the morphism of complexes
$\check C^\bu(M^\bu)\sptilde\rarrow M^\bu$ is a quasi-isomorphism.
 It suffices to check that the cohomology modules of the complex
$\check C^\bu(M^\bu)\sptilde$ are $S$\+torsion $R$\+modules
for every complex of $R$\+modules~$M^\bu$.
 This is so because the complex $S^{-1}R\ot_R\check C^\bu(M^\bu)\sptilde$
is contractible (cf.\ Lemma~\ref{complex-check-c-tilde}(a)).

 To prove that $\sD_{S\tors}^\st(R\modl)\subset\ker(k^*)$, we have to
check that the morphism $\check C^\bu(M^\bu)\sptilde\rarrow M^\bu$
is a quasi-isomorphism for every complex of $R$\+modules $M^\bu$
with $S$\+torsion cohomology modules.
 This is so because the complex $S^{-1}R\ot_RM^\bu$ is acyclic
(cf.\ Lemma~\ref{complex-check-c-tilde}(c)).

 To prove that $\ker(\boR k^!)\subset\sD_{S\ctra}^\st(R\modl)$, it
suffices to check that the cohomology modules of the complex
$\Hom_R(T^\bu,A^\bu)$ are $S$\+contramodule $R$\+modules for every
complex of $R$\+modules~$A^\bu$.
 Every cohomology module of the complex $\Hom_R(T^\bu,A^\bu)$ only
depends on a finite number of terms of the complex~$A^\bu$.
 This reduces the question to the case of a finite complex $A^\bu$.
 Since the full subcategory $\sD_{S\ctra}^\st(R\modl)$ is closed under
shifts and cones in $\sD^\st(R\modl)$, the question further reduces
to the case of a one-term complex $A^\bu=A$.
 Here it remains to apply Lemma~\ref{complex-t}(a).

 To prove that $\sD_{S\ctra}^\st(R\modl)\subset\ker(\boR k^!)$, we have
to check that the morphism $A^\bu\rarrow\Hom_R(T^\bu,A^\bu)$ is
a quasi-isomorphism for every complex of $R$\+modules $A^\bu$ with
$S$\+contramodule cohomology modules.
 It suffices to consider the case when the complex $A^\bu$ is finite,
and the question reduces further to the case of a one-term complex
$A^\bu=A$ (see Hartshorne's lemma on way-out
functors~\cite[Proposition~I.7.1]{Hart}).
 It remains to use Lemma~\ref{complex-t}(c).
\end{proof}

 The following theorem is the most general version of 
a triangulated Matlis equivalence that we are able to obtain.

\begin{thm} \label{triangulated-matlis}
 For any commutative ring $R$ and a multiplicative subset $S\subset R$
such that\/ $\pd_RS^{-1}R\le 1$, and any conventional derived category
symbol\/ $\st=\b$, $+$, $-$, or\/~$\varnothing$, the functor
$k_*\:\sD^\st((S^{-1}R)\modl)\rarrow\sD^\st(R\modl)$ is fully faithful
and its image is a thick subcategory in\/ $\sD^\st(R\modl)$.
 Furthermore, there are natural equivalences of
triangulated categories~\eqref{torsion-contra-cohomology-modules}
$$
 \sD^\st_{S\tors}(R\modl)\simeq\sD^\st(R\modl)/
 \sD^\st((S^{-1}R)\modl)\simeq\sD^\st_{S\ctra}(R\modl).
$$
 The resulting triangulated Matlis equivalence
$$
 \sD^\st_{S\tors}(R\modl)\simeq\sD^\st_{S\ctra}(R\modl)
$$
is provided by the functors taking a complex
$M^\bu\in\sD^\st_{S\tors}(R\modl)$ to the complex
$\Hom_R(T^\bu,M^\bu)\in\sD^\st_{S\ctra}(R\modl)$ and a complex
$A^\bu\in\sD^\st_{S\ctra}(R\modl)$ to the complex
$\check C^\bu(A^\bu)\sptilde\in\sD^\st_{S\tors}(R\modl)$.
\end{thm}

 Notice that, of course, the complex $\check C^\bu(A^\bu)\sptilde$ is
isomorphic to $T^\bu\ot_R A^\bu$ as an object of
$\sD^\st_{S\tors}(R\modl)$ for every complex $A^\bu\in\sD^\st(R\modl)$.

\begin{proof}
 Follows from Proposition~\ref{semiorthogonal-mgm} and
Lemma~\ref{kernel-categories-computed}.
 The first assertion of the theorem is explainable by saying that
$k\:R\rarrow S^{-1}R$ is a homological ring
epimorphism~\cite[Theorem~4.4 and Corollary~4.7(2)]{GL},
\cite[Theorem~3.7]{Pa}.
 Both this and the leftmost one of the two triangulated equivalences do
not depend on the projective dimension assumption $\pd_RS^{-1}R\le1$
(as the construction and the properties of the functor~$k^*$ in
Proposition~\ref{semiorthogonal-mgm} do not require it, and neither
does Lemma~\ref{kernel-categories-computed}(a)).
 The rightmost triangulated equivalence needs $\pd_RS^{-1}R\le1$.
\end{proof}

\begin{rem}
 The following observations, the most part of which the author learned
from the anonymous referee, point out a connection between our
exposition and the infinitely generated tilting/silting theory.
 When all the elements of $S$ are nonzero-divisors in $R$, the ring
homomorphism $k\:R\rarrow S^{-1}R$ is an injective homological ring
epimorphism.
 Hence, according to~\cite[Definition~4.46 and Theorem~14.59]{GT}
or~\cite[Theorem~3.5]{AS}, the direct sum $S^{-1}R\oplus S^{-1}R/R$ is
a $1$\+tilting $R$\+module.
 More generally, according to~\cite[Example~6.5]{MS}, the direct sum
$S^{-1}R\oplus (S^{-1}R)/k(R)$ is a silting $R$\+module, and a $2$\+term
projective resolution of the complex $S^{-1}R\oplus K^\bu$ is
a $2$\+silting complex of $R$\+modules in the sense
of~\cite[Remark~2.7, Proposition~4.2, and Theorem~4.9]{AMT}.
 Equivalently, the complex $S^{-1}R\oplus K^\bu$ is a bounded silting
object of the derived category $\sD(R\modl)$ in the sense
of~\cite[Propositions~4.13 and~4.17]{PV}.
 When there is no $S$\+h-divisible $S$\+torsion in the $R$\+module $R$,
the complex $S^{-1}R\oplus K^\bu$ is even a tilting object in
$\sD(R\modl)$ in the sense of~\cite[Definition~1.1]{PV}.
\end{rem}

\Section{Two Exact Category Equivalences}

 In this section we deduce from Theorem~\ref{triangulated-matlis} our
versions of the Matlis additive category equivalences
of~\cite[\S3]{Mat}.
 In fact, we will even obtain equivalences of exact categories
(in Quillen's sense).
 As in Section~\ref{triangulated-equivalence}, our setting is that
of a commutative ring $R$ with a multiplicative subset $S\subset R$
such that the projective dimension of the $R$\+module $S^{-1}R$
does not exceed~$1$.

\begin{lem} \label{torsion-free-and-h-divisible}
\textup{(a)} An $R$\+module $N$ is $S$\+torsion-free if and only if\/
$\Tor^R_1(K^\bu,N)=0$. \par
\textup{(b)} An $R$\+module $C$ is $S$\+h-divisible if and only if\/
$\Ext_R^1(K^\bu,C)=0$.
\end{lem}

\begin{proof}
 Part~(a) is Lemma~\ref{tor-with-k}(d).
 In part~(b), the ``if'' claim follows from the short
exact sequence~(III).
 To prove the ``only if'', assume that $C$ is $S$\+h-divisible.
 Then from the sequence~(III) we see that the map
$\Ext_R^1(K^\bu,C)\rarrow\Ext_R^1(S^{-1}R,C)$ is an isomorphism.
 Since $\Delta_S(C)=\Ext_R^1(K^\bu,C)$ is an $S$\+contramodule
by Lemma~\ref{delta-produces-contramodule}(c) and
$\Ext_R^1(S^{-1}R,C)$ is an $(S^{-1}R)$\+module, these two $R$\+modules
can only be isomorphic when both of them vanish.
\end{proof}

 It follows from Lemma~\ref{torsion-free-and-h-divisible}(b) that
the full subcategory of $S$\+h-divisible $R$\+modules is closed under
extensions in $R\modl$.
 This provides another proof of Lemma~\ref{h-torsion-theory}(a).

 In particular, the full subcategory of $S$\+h-divisible $S$\+torsion
$R$\+modules in $R\modl$ is closed under extensions, quotients
and infinite direct sums.
 So it inherits an exact category structure from the abelian
category $R\modl$ or $R\modl_{S\tors}$.

 The full subcategory of $S$\+torsion-free $S$\+contramodule
$R$\+modules is closed under extensions, kernels, and infinite
products in $R\modl$; it is also closed under extensions and
subobjects in the abelian category $R\modl_{S\ctra}$.
 So this full subcategory inherits an exact category structure
from the abelian category $R\modl$ or $R\modl_{S\ctra}$.

 The following corollary is our version of~\cite[Theorem~3.4]{Mat}
and~\cite[Corollary~2.4]{Mat2}.

\begin{cor} \label{first-matlis}
 The functors $M\longmapsto\Ext^0_R(K^\bu,M)$ and
$A\longmapsto\Tor_0^R(K^\bu,A)$ establish an equivalence between
the exact categories of $S$\+h-divisible $S$\+torsion $R$\+modules
$M$ and $S$\+torsion-free $S$\+contramodule $R$\+modules~$A$.
\end{cor}

\begin{proof} \hbadness=2850
 Clearly, the equivalence of triangulated categories
$\sD^\b_{S\tors}(R\modl)\simeq\sD^\b_{S\ctra}(R\modl)$ from
Theorem~\ref{triangulated-matlis} restricts to an equivalence
between the exact categories of
\begin{itemize}
\item those complexes in $\sD^\b_{S\tors}(R\modl)$ whose cohomology
are concentrated in the cohomological degree~$0$ and whose images in
$\sD^\b_{S\ctra}(R\modl)$ have cohomology concentrated in degree~$1$,
and
\item those complexes in $\sD^\b_{S\ctra}(R\modl)$ whose cohomology
are concentrated in the cohomological degree~$1$ and whose images
in $\sD^\b_{S\tors}(R\modl)$ have cohomology concentrated in
degree~$0$.
\end{itemize}
 Since the complexes $\check C^\bu(R)\sptilde$ and $T^\bu$ are
quasi-isomorphic to $K^\bu[-1]$, the former category consists of those
$S$\+torsion $R$\+modules $M$ for which $\Ext^1_R(K^\bu,M)=0$, while
the latter category consists of those $S$\+contramodule $R$\+modules
$A$ for which $\Tor_1^R(K^\bu,M)=0$.
 It remains to recall Lemma~\ref{torsion-free-and-h-divisible}.
\end{proof}

 Notice that by
Corollary~\ref{bounded-torsion-sepcomplete-contramodule}
an $S$\+torsion-free $R$\+module $A$ is an $S$\+contramodule if and
only if it is $S$\+separated and $S$\+complete.
 Thus Corollary~\ref{first-matlis} can be formulated as
an equivalence between the additive categories of $S$\+h-divisible
$S$\+torsion $R$\+modules and $S$\+torsion-free $S$\+separated
$S$\+complete $R$\+modules (cf.~\cite[Theorem~VIII.2.8]{FS}).

\medskip

 Our last corollary in this section will involve a class of
$R$\+modules that were called ``adjusted co-torsion'' in~\cite{Harr}
and ``special cotorsion'' in~\cite{Mat}.
 Following the terminology in~\cite{Mat}, we say that an $R$\+module
$N$ is \emph{$S$\+special} if the quotient module $N/\Gamma_S(N)$
is $S$\+divisible.
 If this is the case, $N/\Gamma_S(N)$, being an $S$\+torsion-free
$S$\+divisible $R$\+module, is an $(S^{-1}R)$\+module.

 The next lemma is our version of~\cite[Lemma~3.6]{Mat}
and~\cite[Corollary~1.3]{Mat2}.

\begin{lem} \label{special-and-h-reduced-torsion}
\textup{(a)} An $R$\+module $N$ is $S$\+special if and only if\/
$\Tor^R_0(K^\bu,N)=0$. \par
\textup{(b)} An $R$\+module $C$ has no $S$\+h-divisible $S$\+torsion
if and only if\/ $\Ext_R^0(K^\bu,C)=0$.
\end{lem}

\begin{proof}
 Part~(a): for an $S$\+torsion $R$\+module $M$, one has
$S^{-1}R\ot_RM=0$, hence $\Tor^R_0(K^\bu,M)=0$ in view of
the short exact sequence~(I).
 For an $(S^{-1}R)$\+module $D$ one has $\Tor^R_0(K^\bu,D)=
H^0(K^\bu\ot_R D)=0$.
 Hence one has $\Tor^R_0(K^\bu,N)=0$ for every $S$\+special
$R$\+module~$N$.
 Conversely, if $\Tor^R_0(K^\bu,N)=0$, then $N/\Gamma_R(N)$
is an $(S^{-1}R)$\+module according to the exact sequence~(I).

 Part~(b): the $R$\+module $S^{-1}R/R$ is $S$\+torsion and
$S$\+h-divisible, hence by Lemma~\ref{tor-with-k}(c) one has
$\Ext^0_R(K^\bu,C)=0$ for every $R$\+module $C$ without
$S$\+h-divisible $S$\+torsion.
 Conversely, if $M=\Gamma_S(C)$ then for every morphism
$f\:S^{-1}R\rarrow M$ there exists an element $s\in S$ such
that $sf(1)=0$.
 So $\Hom_R(S^{-1}R/R,C)=0$ implies $sf=0$ and, since $S^{-1}R$
is $S$\+divisible, $f=0$.
 Thus $\Gamma_S(C)$ is $S$\+h-reduced whenever
$\Ext^R_0(K^\bu,C)=0$.
 Unlike the proof of Lemma~\ref{torsion-free-and-h-divisible}(b),
the proof of the present lemma does not depend
on the assumption $\pd_RS^{-1}R\le1$.
\end{proof}

 It follows from Lemma~\ref{special-and-h-reduced-torsion}(a) that
the full subcategory of $S$\+special $R$\+modules is closed under
extensions, quotients, and infinite direct sums in $R\modl$.
 Hence the full subcategory of $S$\+special $S$\+contramodule
$R$\+modules is closed under extensions in $R\modl$; it is
also closed under extensions and quotients in $R\modl_{S\ctra}$.
 So this full subcategory inherits an exact category structure
from the abelian category $R\modl$ or $R\modl_{S\ctra}$.

 The full subcategory of $S$\+h-reduced $S$\+torsion $R$\+modules
is closed under extensions, subobjects, and infinite direct sums
in $R\modl_{S\tors}$ and $R\modl$.
 So it inherits an exact category structure.

 The following corollary is our version of~\cite[Theorem~3.8]{Mat}.

\begin{cor} \label{second-matlis}
 The functors $M\longmapsto\Ext_R^1(K^\bu,M)$ and $A\longmapsto
\Tor^R_1(K^\bu,A)$ establish an equivalence between the exact
categories of $S$\+h-reduced $S$\+torsion $R$\+modules $M$
and $S$\+special $S$\+contramodule $R$\+modules~$A$.
\end{cor}

 We recall that, according to Lemma~\ref{tor-with-k}(d), one has
$\Tor^R_1(K^\bu,A)\simeq\Gamma_S(A)$ and, by the definition,
$\Ext_R^1(K^\bu,M)=\Delta_S(M)$.

\begin{proof}
 The equivalence of triangulated categories $\sD^\b_{S\tors}(R\modl)
\simeq\sD^\b_{S\ctra}(R\modl)$ from Theorem~\ref{triangulated-matlis}
restricts to an equivalence between the exact categories of
\begin{itemize}
\item those complexes in $\sD^\b_{S\tors}(R\modl)$ whose cohomology
are concentrated in the cohomological degree~$0$ and whose images
in $\sD^\b_{S\ctra}(R\modl)$ have cohomology concentrated in
degree~$0$, and
\item those complexes in $\sD^\b_{S\ctra}(R\modl)$ whose cohomology
are concentrated in the cohomological degree~$0$ and whose images
in $\sD^\b_{S\tors}$ have cohomology concentrated in degree~$0$.
\end{itemize}
 Due to the quasi-isomorphisms $\check C^\bu(R)\sptilde=
K^\bu[-1]\simeq T^\bu$, the former category consists of those
$S$\+torsion $R$\+modules $M$ for which $\Ext_R^0(K^\bu,M)=0$,
while the latter category consists of those $S$\+contramodule
$R$\+modules $A$ for which $\Tor^R_0(K^\bu,A)=0$.
 It remains to take account of
Lemma~\ref{special-and-h-reduced-torsion}.
\end{proof}

 In any $R$\+module $M$, there is a unique maximal $S$\+h-divisible
$S$\+torsion submodule $\theta_S(M)\subset M$, which can be constructed
as the sum of the images of all the $R$\+module morphisms
$S^{-1}R/R\rarrow M$.
 The quotient module $M/\theta_S(M)$ is the (unique) maximal quotient
$R$\+module of $M$ having no $S$\+h-divisible $S$\+torsion.
 When $M$ is an $S$\+torsion $R$\+module, its submodule $\theta_S(M)$
belongs to the exact subcategory appearing in
Corollary~\ref{first-matlis} and the quotient module $M/\theta_S(M)$
belongs to the exact subcategory appearing in
Corollary~\ref{second-matlis}.

 In any $R$\+module $C$, there is a unique maximal $S$\+special
$R$\+submodule $\sigma_S(C)\subset C$, which can be constructed as
the sum of all the $S$\+special $R$\+submodules in~$C$.
 The quotient module $C/\sigma_S(C)$ is the (unique) maximal
$S$\+h-reduced (or $S$\+reduced) $S$\+torsion-free quotient $R$\+module
of~$C$.
 When $C$ is an $S$\+contramodule, the $R$\+module $\sigma_S(C)$ is
an $S$\+contramodule by Lemma~\ref{contramodule-short-exact}(b),
and the $R$\+module $C/\sigma_S(C)$ is an $S$\+contramodule by
Lemma~\ref{contramodule-short-exact}(c)
(this is our version of~\cite[Theorem~3.7]{Mat}).
 So the submodule $\sigma_S(C)\subset C$ belongs to the subcategory
appearing in Corollary~\ref{second-matlis} and the quotient module
$C/\sigma_S(C)$ belongs to the subcategory appearing in
Corollary~\ref{first-matlis}.

\begin{rem}
 Let us emphasize that our results in this section are both more
and less general than in Matlis'~\cite[\S3]{Mat} (also
in~\cite[Section~VIII.2]{FS}).
 On the one hand, in place of a commutative domain $R$ with
the multiplicative subset $R\setminus\{0\}$, we have a rather
arbitrary  commutative ring $R$ and a multiplicative subset
$S\subset R$.
 The ease with which replacing the quotient module $K=Q/R$ by
a two-term complex~$K^\bu$ allows to work with zero-divisors in
this theory is remarkable.
 On the other hand, our triangulated equivalence seems to be unable
to avoid the assumption $\pd_RS^{-1}R\le 1$, which was not made
in the classical approach.
\end{rem}

\Section{Two Fully Faithful Triangulated Functors}

 Let $R$ be a commutative ring and $S\subset R$ be a multiplicative
subset.
 In addition to our running assumption that $\pd_SS^{-1}R\le 1$, in
this section we will also assume that the $S$\+torsion in $R$ is
bounded, that is there exists $r\in S$ such that $r\Gamma_S(R)=0$.
 Our aim is to rewrite the triangulated equivalence of
Theorem~\ref{triangulated-matlis} as an equivalence between
the derived categories of the abelian categories $R\modl_{S\tors}$
and $R\modl_{S\ctra}$ (cf.~\cite[Section~1 and~2]{Pmgm}).

\begin{lem} \label{bounded-torsion-injective-flat}
\textup{(a)} For every injective $R$\+module $J$ one has\/
$\Tor^R_0(K^\bu,J)=0$. \par
\textup{(b)} For every flat $R$\+module $F$ one has\/
$\Ext_R^0(K^\bu,F)=0$.
\end{lem}

\begin{proof}
 We assume that $r\Gamma_S(R)=0$, where $r\in S$.
 Part~(a): let us first show that for every $s\in S$ one has $srJ=rJ$.
 Indeed, for every $t\in R$ elements of the submodule $tJ\subset J$
correspond to those $R$\+module morphisms $R\rarrow J$ that factorize
through the surjection $t\:R\rarrow tR$ (since every morphism $R\supset
tR\rarrow J$ can be extended to a morphism $R\rarrow J$).
 By assumption, every element annihilated by $sr$ in $R$ is also
annihilated by~$r$.
 Therefore, the restriction $s\:rR\rarrow R$ of the map $s\:R\rarrow R$
to the submodule $rR\subset R$ is injective, and the map
$s\:rR\rarrow srR$ is an isomorphism.
 It follows that every morphism $R\rarrow J$ that factorizes through
the surjection $r\:R\rarrow rR$ also factorizes through the surjection
$sr\:R\rarrow srR$, that is $srJ=rJ$.

 Now for every $j\in J$ there exists $j'\in J$ such that $rj=rsj'$,
hence $j/s=rj/rs=rsj'/rs=j'$ in $S^{-1}J$.
 We have shown that the map $J\rarrow S^{-1}J$ is surjective, and
the vanishing of $\Tor^R_0(K^\bu,J)$ follows by means of
the exact sequence~(I).

 Part~(b): as it was explained in the proof of
Corollary~\ref{flat-bounded-torsion}, one has $r\Gamma_S(F)=0$.
 Hence the $R$\+module $\Gamma_S(F)$ is $S$\+h-reduced, and it
remains to use Lemma~\ref{special-and-h-reduced-torsion}(b).
\end{proof}

 The next two lemmas follow straightforwardly from
Lemmas~\ref{complex-check-c-tilde}(b)\+-\ref{complex-t}(b)
and Lemma~\ref{bounded-torsion-injective-flat}.

\begin{lem} \label{gamma-derived}
\textup{(a)} The complex $\check C^\bu(M)\sptilde$ assigned to
an $R$\+module $M$ computes the right derived functor\/
$\boR^*\Gamma_S(M)$ of the left exact functor\/ $\Gamma_S\:
R\modl\rarrow R\modl_{S\tors}$ viewed as taking values in
the ambient category $R\modl$. \par
\textup{(b)} The right derived functor\/ $\boR^*\Gamma_S(M)$ of
the left exact functor\/ $\Gamma_S\:R\modl\rarrow R\modl_{S\tors}$
has homological dimension~$\le1$.
\end{lem}

\begin{lem} \label{delta-derived}
\textup{(a)} The complex\/ $\Hom_R(T^\bu,C)$ assigned to an $R$\+module
$C$ computes the left derived functor\/ $\boL_*\Delta_S(C)$ of
the right exact functor\/ $\Delta_S\:R\modl\rarrow R\modl_{S\ctra}$
viewed as taking values in the ambient category $R\modl$. \par
\textup{(b)} The left derived functor\/ $\boL_*\Delta_S(C)$ of
the right exact functor\/ $\Delta_S\:R\modl\rarrow R\modl_{S\ctra}$
has homological dimension~$\le1$.
\end{lem}

\begin{proof}
 The proofs of Lemmas~\ref{gamma-derived} and~\ref{delta-derived}
are very similar.
 Let us explain Lemma~\ref{delta-derived}.
 First of all, the functor $\Delta_S=\Ext^1_R(K^\bu,{-})$ is right
exact, because $\Ext^2_R(K^\bu,{-})=0$ by
Lemma~\ref{tor-with-k}(b) (or because $\Delta_S\:R\modl\rarrow
R\modl_{S\ctra}$ is a left adjoint functor to the exact embedding
functor $R\modl_{S\ctra}\rarrow R\modl$).
 Notice that $\Hom_R(T^\bu,C)$ is a complex in $R\modl$ and not in
$R\modl_{S\ctra}$; hence the caveat about it computing the derived
functor $\boL_*\Delta(C)$ viewed as taking values in
$R\modl$ rather than $R\modl_{S\ctra}$.
 The embedding functor $R\modl_{S\ctra}\rarrow R\modl$ being, however,
fully faithful and, in particular, taking nonzero objects to nonzero
objects, part~(b) follows from part~(a).

 To prove part~(a), consider a left projective $R$\+module
resolution $F_\bu$ of the $R$\+module~$C$.
 We have to show that the complexes $\Delta_S(F_\bu)$ and
$\Hom_R(T^\bu,C)$ are connected by a chain of quasi-isomorphisms
of complexes of $R$\+modules.
 Indeed, by Lemma~\ref{complex-t}(b) we have a natural isomorphism
$H_0\Hom_R(T^\bu,C)=\Delta_S(C)$.
 Consider the total complex of the bicomplex
$\Hom_R(T^\bu,F_\bu)$.
 Then there are natural morphisms of complexes of $R$\+modules
$$
 \Delta_S(F_\bu)\llarrow\Hom_R(T^\bu,F_\bu)\lrarrow
 \Hom_R(T^\bu,C).
$$
 Now $T^\bu$ is a finite complex of projective $R$\+modules, so
the quasi-isomorphism $F_\bu\rarrow C$ induces a quasi-isomorphism
$\Hom_R(T^\bu,F_\bu)\rarrow\Hom_R(T^\bu,C)$.
 On the other hand, for any $R$\+module $C$ one has
$H_1\Hom_R(T^\bu,C)=\Ext^0_R(K^\bu,C)$, hence the morphism
$\Hom_R(T^\bu,F_\bu)\rarrow\Delta_S(F_\bu)$ is a quasi-isomorphism by
Lemma~\ref{bounded-torsion-injective-flat}(b).
\end{proof}

 The following theorem is essentially proved in~\cite[Theorem~1.3
or Theorem~2.9]{Pmgm}.
 Let $\sA$ be an abelian category with enough projective objects
and $\sC\subset\sA$ be a full subcategory closed under the kernels,
cokernels, and extensions; so $\sC$ is an abelian category and
the embedding functor $\sC\rarrow\sA$ is exact.
 Let $\Delta\:\sA\rarrow\sC$ be a functor left adjoint to the fully
faithful embedding functor $\sC\rarrow\sA$.

\begin{thm} \label{triangulated-fully-faithful}
 Assume that the left derived functor\/ $\boL_*\Delta$ of
the right exact functor\/ $\Delta$ has finite homological dimension,
that is there exists $d\ge0$ such that\/ $\boL_n\Delta(A)=0$ for
all $A\in\sA$ and all\/ $n>d$.
 Assume further that\/ $\boL_n\Delta(C)=0$ for every object
$C\in\sC\subset\sA$ and all\/ $n>0$.
 Then for every derived category symbol\/ $\st=\b$, $+$, $-$,
$\varnothing$, $\abs+$, $\abs-$, or\/~$\abs$, the exact embedding
functor\/ $\sC\rarrow\sA$ induces a fully faithful triangulated
functor
\begin{equation} \label{induced-triangulated}
 \sD^\st(\sC)\lrarrow\sD^\st(\sA).
\end{equation}
\end{thm}

\begin{proof}[Sketch of proof]
 Denote by $\sA_{\Delta\adj}\subset\sA$ the full subcategory of all
objects $A\in\sA$ such that $\boL_n\Delta(A)=0$ for all $n>0$.
 Then the full subcategory $\sA_{\Delta\adj}$ is closed under extensions
and the kernels of epimorphisms in $\sA$; in particular,
$\sA_{\Delta\adj}$ inherits an exact category structure from
the abelian category~$\sA$.
 Furthermore, every object of $\sA$ admits a finite left resolution
of length~$d$ by objects of $\sA_{\Delta\adj}$.
 It follows that the functor $\sD^\st(\sA_{\Delta\adj})\rarrow
\sD^\st(\sA)$ induced by the exact embedding
$\sA_{\Delta\adj}\rarrow\sA$ is an equivalence of triangulated
categories~\cite[Proposition~A.5.6]{Pcosh}.
 The functor $\Delta$ restricted to the exact subcategory
$\sA_{\Delta\adj}\subset\sA$ is exact.
 Applying $\Delta$ to complexes of objects from $\sA_{\Delta\adj}$, one
constructs a triangulated functor
$$
 \boL\Delta\:\sD^\st(\sA)\lrarrow\sD^\st(\sC),
$$
which is left adjoint to the functor~\eqref{induced-triangulated}.
 Now the composition $\sD^\st(\sC)\rarrow\sD^\st(\sA)\lrarrow
\sD^\st(\sC)$ is the identity functor, since $\sC\subset
\sA_{\Delta\adj}$ by an assumption of the theorem and the composition
$\sC\rarrow\sA\rarrow\sC$ is the identity (as the functor
$\sC\rarrow\sA$ is fully faithful).
 It follows that the functor~\eqref{induced-triangulated} is fully
faithful.
\end{proof}

 Furthermore, the following proposition is essentially proved
in~\cite[proof of Corollary~1.4 or Corollary~2.10]{Pmgm}.

\begin{prop} \label{essential-image-identified}
 In the assumptions of Theorem~\ref{triangulated-fully-faithful},
for every conventional derived category symbol\/ $\st=\b$, $+$, $-$,
or\/~$\varnothing$, the essential image of the triangulated
functor~\eqref{induced-triangulated} coincides with the full
triangulated subcategory
$$
 \sD^\st_\sC(\sA)\.\subset\.\sD^\st(\sA)
$$
of all complexes in\/ $\sD^\st(\sA)$ with the cohomology objects
belonging to\/~$\sC$. \qed
\end{prop}

\begin{thm} \label{bounded-torsion-essential-images}
 Let $R$ be a commutative ring and $S\subset R$ be a multiplicative
subset such that\/ $\pd_RS^{-1}R\le 1$ and the $S$\+torsion in $R$
is bounded.  Then for every conventional derived category symbol\/
$\st=\b$, $+$, $-$, or\/~$\varnothing$ \par
\textup{(a)} the triangulated functor between the derived categories
$$
 \sD^\st(R\modl_{S\tors})\lrarrow\sD^\st(R\modl)
$$
induced by the exact embedding functor $R\modl_{S\tors}\rarrow
R\modl$ is fully faithful, and its essential image coincides with
the full subcategory
$$
 \sD^\st_{S\tors}(R\modl)\.\subset\.\sD^\st(R\modl),
$$
providing an equivalence of triangulated
categories~\eqref{essential-images}
$$
 \sD^\st(R\modl_{S\tors})\.\simeq\.\sD^\st_{S\tors}(R\modl);
$$ \par
\textup{(b)} the triangulated functor between the derived categories
$$
 \sD^\st(R\modl_{S\ctra})\lrarrow\sD^\st(R\modl)
$$
induced by the exact embedding functor $R\modl_{S\ctra}\rarrow
R\modl$ is fully faithful, and its essential image coincides with
the full subcategory
$$
 \sD^\st_{S\ctra}(R\modl)\.\subset\.\sD^\st(R\modl),
$$
providing an equivalence of triangulated
categories~\eqref{essential-images}
$$
 \sD^\st(R\modl_{S\ctra})\.\simeq\.\sD^\st_{S\ctra}(R\modl).
$$
\end{thm}

\begin{proof}
 To prove part~(b), one applies
Theorem~\ref{triangulated-fully-faithful}
and Proposition~\ref{essential-image-identified}.
 To prove part~(a), one applies the opposite versions of
Theorem~\ref{triangulated-fully-faithful}
and Proposition~\ref{essential-image-identified}.
 In both cases, it remains to check that the assumptions of
Theorem~\ref{triangulated-fully-faithful} (or the opposite
assumptions) are satisfied.

 The condition of finite homological dimension of
the derived functor $\boL_*\Delta_S$ (in the case of part~(b))
or of the derived functor $\boR^*\Gamma_S$ (in the case of
part~(a)) holds by
Lemmas~\ref{gamma-derived}(b)\+-\ref{delta-derived}(b).
 To show that $\boL_n\Delta_S(C)=0$ for every $S$\+contramodule
$R$\+module $C$ and all $n>0$ (which means, actually, $n=1$),
one compares Lemma~\ref{delta-derived}(a)
with Lemma~\ref{complex-t}(c).
 Similarly, to show that $\boR^n\Gamma_S(M)=0$ for every
$S$\+torsion $R$\+module $M$ and all $n>0$ (i.~e., $n=1$)
one compares Lemma~\ref{gamma-derived}(a) with
Lemma~\ref{complex-check-c-tilde}(c).
\end{proof}

 The next corollary is the main result of this section.

\begin{cor} \label{equivalence-of-derived-categories}
 For every commutative ring $R$ with a multiplicative subset
$S\subset R$ such that\/ $\pd_RS^{-1}R\le 1$ and the $S$\+torsion in $R$
is bounded, and for every conventional derived category symbol\/
$\st=\b$, $+$, $-$, or\/~$\varnothing$, there is a natural triangulated
equivalence~\eqref{equivalence-between-derived} between the derived
categories of the abelian categories $R\modl_{S\tors}$ and
$R\modl_{S\ctra}$ of $S$\+torsion and $S$\+contramodule $R$\+modules
$$
 \sD^\st(R\modl_{S\tors})\simeq\sD^\st(R\modl_{S\ctra}). 
$$
\end{cor}

\begin{proof}
 Compare Theorem~\ref{triangulated-matlis} with
Theorem~\ref{bounded-torsion-essential-images}.
\end{proof}

\begin{rem}
 Both parts~(a) and~(b) of
Theorem~\ref{bounded-torsion-essential-images}
remain true under somewhat weaker assumptions.
 Namely, the assumption that $\pd_RS^{-1}R\le 1$ was not used in
the proof of Theorem~\ref{bounded-torsion-essential-images}(a), so
it is not relevant for its validity.
 On the other hand, the full strength of the assumption that
the $S$\+torsion in $R$ is bounded is not necessary for the proof
of Theorem~\ref{bounded-torsion-essential-images}(b), as it suffices
to require that there be no $S$\+h-divisible $S$\+torsion in
the $R$\+module~$R$ for the purposes of part~(b).
 Indeed, the only place where the restriction on the $S$\+torsion
in $R$ was used in the proof of
Theorem~\ref{bounded-torsion-essential-images}(b) was in
Lemma~\ref{bounded-torsion-injective-flat}(b); and one easily observes
that $\Ext_R^0(K^\bu,F)=0$ for all projective $R$\+modules $F$
whenever there is no $S$\+h-divisible $S$\+torsion in~$R$.

 The assumption that there be no $S$\+h-divisible $S$\+torsion in $R$
is not sufficient for the validity of
Theorem~\ref{bounded-torsion-essential-images}(a), though, as one can
see using the following argument.
 Let $R$ be a commutative ring and $S\subset R$ be a multiplicative
subset.
 Then, for any injective $R$\+module $J$, the functor right adjoint to
the triangulated functor $\sD^\b(R\modl_{S\tors})\rarrow\sD^\b(R\modl)$
is defined on the object $J\in\sD^\b(R\modl)$ and takes it to the object
$\Gamma_S(J)\in\sD^\b(R\modl_{S\tors})$.
 On the other hand, the functor right adjoint to the embedding functor
$\sD^\b_{S\tors}(R\modl)\rarrow\sD^\b(R\modl)$ takes $J$ to the object
$\check C^\bu(M)\sptilde\in\sD^\b_{S\tors}(R\modl)$.
 Now if the functor $\sD^\b(R\modl_{S\tors})\rarrow\sD^\b(R\modl)$ is
fully faithful, then it is clear that its essential image coincides
with the full subcategory $\sD^\b_{S\tors}(R\modl)\subset\sD^\b(R\modl)$,
and it follows that the objects $\Gamma_S(J)$ and
$\check C^\bu(M)\sptilde$ must be isomorphic in
$\sD^\b_{S\tors}(R\modl)$, that is the assertion of
Lemma~\ref{bounded-torsion-injective-flat}(a) has to hold.
 When the multiplicative subset $S\subset R$ is generated by one
element $s\in R$, this condition is equivalent to boundedness of
$S$\+torsion in~$R$.
\end{rem}

\Section{The Dedualizing Complex}

 The proof of Corollary~\ref{equivalence-of-derived-categories} still
leaves something to be desired.
 Given a complex of $S$\+torsion $R$\+modules $M^\bu$, in order to
obtain the related complex of $S$\+contramodule $R$\+modules
$C^\bu$ following this proof, one would have to consider the complex
of $R$\+modules
$$
 A^\bu=\Hom_R(T^\bu,M^\bu)\in\sD^\st_{S\ctra}(R\modl)
$$
with $S$\+contramodule cohomology $R$\+modules, and then pass from it
to a complex of $S$\+contramodule $R$\+modules $C^\bu$ by applying
the functor $\boL\Delta_S$, that is
$$
 C^\bu=\boL\Delta_S(A^\bu).
$$
 The terms of the complex $A^\bu$ themselves are \emph{not}
$S$\+contramodules.
 One can avoid the preliminary step of this two-step procedure by
setting directly
$$
 C^\bu=\boL\Delta_S(M^\bu),
$$
but in order to apply the derived functor $\boL\Delta_S$ to
an $S$\+torsion $R$\+module or a complex of $S$\+torsion $R$\+modules
one would still have to use resolutions in the category of arbitrary
$R$\+modules $R\modl$.

 Similarly, given a complex of $S$\+contramodule $R$\+modules $C^\bu$,
in order to obtain the related complex of $S$\+torsion $R$\+modules
$M^\bu$ one would have to consider the complex of $R$\+modules
$$
 N^\bu=\check C^\bu(C^\bu)\sptilde\ \text{or}\ T^\bu\ot_R C^\bu
 \.\in\.\sD^\st_{S\tors}(R\modl)
$$
with $S$\+torsion cohomology $R$\+modules, and then pass from it to
a complex of $S$\+torsion $R$\+modules $M^\bu$ by applying
the functor $\boR\Gamma_S$, that is
$$
 M^\bu=\boR\Gamma_S(N^\bu).
$$
 The terms of the complex $N^\bu$ themselves are \emph{not}
$S$\+torsion $R$\+modules.
 One can avoid the preliminary step of this procedure by setting
$$
 M^\bu=\boR\Gamma_S(C^\bu),
$$
but in order to apply the derived functor $\boR\Gamma_S$ to
an $S$\+contramodule $R$\+module or a complex of $S$\+contramodule
$R$\+modules one would still have to use resolutions in the category
of arbitrary $R$\+modules.

 We would like to have direct constructions of the mutually inverse
triangulated functors
$$
 \sD^\st(R\modl_{S\tors})\lrarrow\sD^\st(R\modl_{S\ctra})
 \!\!\quad\text{and}\quad\!\!
 \sD^\st(R\modl_{S\ctra})\lrarrow\sD^\st(R\modl_{S\tors})
$$
staying entirely inside the two covariantly dual worlds of
$S$\+torsion and $S$\+contra\-module $R$\+modules and never involving
$R$\+modules of more general nature.
 The test assertions or constructions we want to be able to
demonstrate with such a technique are the equivalences of the absolute
derived categories of $S$\+torsion and $S$\+contra\-modules
$R$\+modules~\eqref{equivalence-between-derived} with
$\st=\abs+$, \,$\abs-$, or~$\abs$.

 In this section, a partial solution for this problem is obtained:
we indeed construct mutually inverse functors acting directly
between the derived categories of $S$\+torsion and $S$\+contramodule
$R$\+modules without going through complexes of arbitrary
$S$\+modules, and we indeed obtain
the equivalences~\eqref{equivalence-between-derived} for absolute
derived categories.
 However, the \emph{proof} of the assertion that our functors are
mutually inverse equivalences involves $R$\+modules not belonging
to $R\modl_{S\tors}$ or $R\modl_{S\ctra}$.

\medskip

 Let $S$ be a multiplicative subset in a commutative ring~$R$.
 We keep assuming that $\pd_RS^{-1}R\le 1$ and the $S$\+torsion in
$R$ is bounded.
 Lemma~\ref{t-s-nakayama} and Proposition~\ref{compact-generators}
will not depend on these assumptions yet.

 For every element $s\in S$, we denote by $T_s^\bu$ the two-term
complex $R\overset s\rarrow R$, with the two terms $R$ sitting in
the cohomological degrees~$0$ and~$1$.
 Notice that the complex $T_s^\bu$ is almost self-dual: one has
$\Hom_R(T_s^\bu,R)\simeq T_s^\bu[1]$.
 Hence for every complex of $R$\+modules $A^\bu$ there is
an isomorphism $\Hom_R(T_s^\bu,A^\bu)\simeq T_s^\bu[1]\ot_R A^\bu$.

\begin{lem} \label{t-s-nakayama}
 Let $A^\bu$ be a complex of $R$\+modules such that either \par
\textup{(a)} the cohomology modules $H^*(A^\bu)$ are $S$\+torsion
$R$\+modules, or \par
\textup{(b)} the cohomology modules $H^*(A^\bu)$ are $S$\+contramodules.
\par
 Suppose that the complex $T_s^\bu\ot_R A^\bu$ is acyclic for every
$s\in S$.
 Then the complex $A^\bu$ is acyclic.
\end{lem}

\begin{proof}
 For every complex of $R$\+modules $A^\bu$, there are natural short
exact sequences
$$
 0\lrarrow H^{n-1}(A^\bu)/sH^{n-1}(A^\bu)\lrarrow
 H^n(T^\bu\ot_RA^\bu)\lrarrow{}_sH^n(A^\bu)\lrarrow0,
$$
where we denote by ${}_sM\subset M$ the submodule of elements
annihilated by $s$ in an $R$\+module~$M$.
 So $H^*(T^\bu\ot_RA^\bu)=0$ means that $s$~acts invertibly in
$H^*(A^\bu)$.
 If this holds for every $s\in S$ then the cohomology modules
$H^*(A^\bu)$ are $(S^{-1}R)$\+modules.
 Now any $S$\+torsion $R$\+module that is similtaneously
an $(S^{-1}R)$\+module vanishes, and so does any $S$\+contramodule
$R$\+module that is simultaneously an $(S^{-1}R)$\+module.
\end{proof}

 The exposition in this section follows the lines
of~\cite[Section~5]{Pmgm}.
 In order to make the analogy more explicit, we formulate
the following version of~\cite[Proposition~5.1]{Pmgm}.

\begin{prop} \label{compact-generators}
 The collection of all complexes $T_s^\bu$, \,$s\in S$ is a set of
compact generators of the triangulated category\/
$\sD_{S\tors}(R\modl)$.
 A complex of $R$\+modules with $S$\+torsion cohomology modules is
a compact object of\/ $\sD_{S\tors}(R\modl)$ if and only if it is
a compact object of\/ $\sD(R\modl)$.
\end{prop}

\begin{proof}
 Both assertions follows from Lemma~\ref{t-s-nakayama}(a) together
with the observations that the complexes $T_s^\bu$ belong to
$\sD_{S\tors}(R\modl)$ and are compact in $\sD(R\modl)$.
\end{proof}

 The following theorem is our version of~\cite[Remark~5.6]{Pmgm}.
 The assumption of boundedness of $S$\+torsion in $R$ is essential
for its validity.

\begin{thm} \label{ext-tor-bounded-torsion-thm}
\textup{(a)} Let $M$ be an $S$\+torsion $R$\+module and $E$ be
an injective object of $R\modl_{S\tors}$.
 Then\/ $\Ext^n_R(M,E)=0$ for all\/ $n>0$. \par
\textup{(b)} Let $C$ be an $S$\+contramodule $R$\+module and $P$
be a projective object of $R\modl_{S\ctra}$.
 Then\/ $\Ext^n_R(P,C)=0$ for all\/ $n>0$. \par
\textup{(c)} Let $M$ be an $S$\+torsion $R$\+module and $P$ be
a projective object of $R\modl_{S\ctra}$.
 Then\/ $\Tor^R_n(M,P)=0$ for all\/ $n>0$.
\end{thm}

\begin{proof}
 Part~(a) is a particular case of the assertion that the functor
$\sD^\b(R\modl_{S\tors})\allowbreak\rarrow\sD^\b(R\modl)$ is fully
faithful (see Theorem~\ref{bounded-torsion-essential-images}(a)).
 Part~(b) is a particular case of the assertion that the functor
$\sD^\b(R\modl_{S\ctra})\rarrow\sD^\b(R\modl)$ is fully faithful
(see Theorem~\ref{bounded-torsion-essential-images}(b)).
 To deduce part~(c) from part~(b), choose an injective $R$\+module~$J$.
 Then $\Hom_R(M,J)$ is an $S$\+contramodule $R$\+module by
Lemma~\ref{hom-contramodule}, hence $\Hom_R(\Tor^R_n(M,P),J)\simeq
\Ext_R^n(P,\Hom_R(M,J))=0$ for $n>0$.
\end{proof}

 One can define \emph{a dedualizing complex of $S$\+torsion
$R$\+modules} by the list of three conditions very similar to
the conditions~(i\+iii) of~\cite[Section~5]{Pmgm}.
 In this section, we follow a simpler path.
 The complex $K^\bu[-1]$ has $S$\+torsion cohomology modules, hence,
according to Theorem~\ref{bounded-torsion-essential-images}(a),
there exists a finite complex of $S$\+torsion $R$\+modules $B^\bu$
quasi-isomorphic to~$K^\bu[-1]$.
 Obviously, one can assume $B^\bu$ to be a two-term complex concentrated
in the cohomological degrees~$0$ and~$1$.
 We choose such a complex of $S$\+torsion $R$\+modules $B^\bu$, and
call it \emph{the dedualizing complex} for the ring $R$ and
the multiplicative subset $S\subset R$.

 Furthermore, we set $B_s^\bu=T_s^\bu\ot_RB^\bu$.
 Clearly, $B^\bu_s$ is a finite complex of $S$\+torsion $R$\+modules
quasi-isomorphic to the complex of $R$\+modules $T_s^\bu\ot_RK^\bu[-1]$,
which is quasi-isomorphic to~$T_s^\bu$.
 So we can assume to have chosen a quasi-isomorphism of complexes
of $R$\+modules $b_s\:T_s^\bu\rarrow B_s^\bu$.

 The following lemma is our version of~\cite[Lemma~5.4(b\+c)]{Pmgm}.

\begin{lem} \label{projective-tensor-lemma}
\textup{(a)} For every projective $R$\+module $F$ and every element
$s\in S$, the morphism of complexes of $R$\+modules $T_s^\bu\ot_R F
\rarrow T_s^\bu\ot_R\Delta_S(F)$ induced by the morphism of
$R$\+modules\/ $\delta_{S,F}\:F\rarrow\Delta_S(F)$ is
a quasi-isomorphism. \par
\textup{(b)} For every projective object $P\in R\modl_{S\ctra}$ and
every element $s\in S$, the morphism of complexes of $R$\+modules
$T_s^\bu\ot_RP\rarrow B_s^\bu\ot_RP$ induced by the morphism 
$b_s\:T^\bu\rarrow B_s^\bu$ is a quasi-isomorphism.
\end{lem}

\begin{proof}
 Part~(a): according to the short exact sequences~(II\+-III) and
Lemma~\ref{special-and-h-reduced-torsion}(b), for every $R$\+module
$F$ without $S$\+h-divisible $S$\+torsion the cohomology modules
of the two-term complex $F\rarrow\Delta_S(F)$ are $(S^{-1}R)$\+modules.
 Hence the tensor product complex $T_s^\bu\ot_R(F\to\Delta_S(F))$ is
acyclic.

 Part~(b): according to Theorem~\ref{ext-tor-bounded-torsion-thm}(c),
both the complexes $T_s^\bu\ot_R P$ and $B_s^\bu\ot_RP$ compute
the derived tensor product $T_s^\bu\ot^\boL_RP=B_s^\bu\ot^\boL_RP$.
 Alternatively, it suffices to consider the case of $P=\Delta_S(F)$,
where $F$ is a projective $R$\+module (as all the projective objects
in $R\modl_{S\ctra}$ are direct summands of modules of this form).
 We have a commutative diagram
$$
\begin{diagram}
\node{T_s^\bu\ot_R F}\arrow{s}\arrow{e}\node{B_s^\bu\ot_RF}
\arrow{s} \\
\node{T_s^\bu\ot_R\Delta_S(F)}\arrow{e}\node{B_s^\bu\ot_R\Delta_S(F)}
\end{diagram}
$$
 The morphism $T_s^\bu\ot_R F\rarrow B_s^\bu\ot_RF$ is
a quasi-isomorphism since $F$ is a flat $R$\+module.
 The morphism $B_s^\bu\ot_RF\rarrow B_s^\bu\ot_F\Delta_S(F)$ is
an isomorphism of complexes by Lemma~\ref{delta-mod-s}, because
$B_s^\bu$ is a complex of $S$\+torsion $R$\+modules.
 The morphism $T_s^\bu\ot_RF\rarrow T_s^\bu\ot_R\Delta_S(F)$ is
a quasi-isomorphism by part~(a).
 It follows that the fourth morphism in the diagram is also
a quasi-isomorphism.
\end{proof}

 The next lemma is our version of~\cite[Lemma~5.5]{Pmgm}.

\begin{lem} \label{injective-hom-lemma}
\textup{(a)} For every injective $R$\+module $J$ and every element
$s\in S$, the morphism of complexes of $R$\+modules $T_s^\bu\ot_R
\Gamma_S(J)\rarrow T_s^\bu\ot_R J$ induced by the embedding
of $R$\+modules\/ $\gamma_{S,J}\:\Gamma_S(J)\rarrow J$ is
a quasi-isomorphism. \par
\textup{(b)} For every injective object $E\in R\modl_{S\tors}$ and
every element $s\in S$, the morphism of complexes of $R$\+modules
$\Hom_R(B_s^\bu,E)\rarrow\Hom_R(T_s^\bu,E)$ induced by the morphism
$b_s\:T_s^\bu\rarrow B_s^\bu$ is a quasi-isomorphism.
\end{lem}

\begin{proof}
 Part~(a): According to Lemma~\ref{bounded-torsion-injective-flat}(a)
and the short exact sequence~(I), for every injective $R$\+module $J$
the quotient module $J/\Gamma_S(J)$ is an $(S^{-1}R)$\+module.
 Hence the complex $T_s^\bu\ot_R J/\Gamma_S(J)$ is acyclic.

 Part~(b): according to Theorem~\ref{ext-tor-bounded-torsion-thm}(a),
both the complexes $\Hom_R(B_s^\bu,E)$ and $\Hom_R(T_s^\bu,E)$ compute
the derived category object $\boR\Hom_R(B_s^\bu,E)=
\boR\Hom_R(T_s^\bu,E)$.
 Alternatively, it suffices to consider the case of $E=\Gamma_S(J)$,
where $J$ is an injective $R$\+module.
 We have a commutative diagram \hfuzz=2.5pt
$$
\begin{diagram}
\node{\Hom_R(B_s^\bu,\Gamma_S(J))}\arrow{s}\arrow{e}
\node{\Hom_R(T_s^\bu,\Gamma_S(J))}\arrow{s} \\
\node{\Hom_R(B_s^\bu,J)}\arrow{e}\node{\Hom_R(T_s^\bu,J)}
\end{diagram}
$$
 The morphism $\Hom_R(B_s^\bu,J)\rarrow\Hom_R(T_s^\bu,J)$ is
a quasi-isomorphism since $J$ is an injective $R$\+module.
 The morphism $\Hom_R(B_s^\bu,\Gamma_S(J))\rarrow\Hom_R(B_s^\bu,J)$
is an isomophism of complexes since $B_s^\bu$ is a complex of
$S$\+torsion $R$\+modules.
 The morphism $\Hom_R(T_s^\bu,\Gamma_S(J))\rarrow
\Hom_R(T_s^\bu,J)$ is a quasi-isomorphism by part~(a).
 It follows that the fourth morphism in the diagram is also
a quasi-isomorphism.
\end{proof}

 The following theorem is our main result.

\begin{thm}
 Let $S$ be a multiplicative subset in a commutative ring~$R$.
 Assume that $\pd_RS^{-1}R\le 1$ and the $S$\+torsion in $R$ is
bounded.
 Then for every derived category symbol\/ $\st=\b$, $+$, $-$,
$\varnothing$, $\abs+$, $\abs-$, or\/~$\abs$ there is an equivalence
of derived categories\/ $\sD^\st(R\modl_{S\tors})\simeq
\sD^\st(R\modl_{S\ctra})$ provided by mutually inverse functors\/
$\boR\Hom_R(B^\bu,{-})$ and $B^\bu\ot_R^\boL{-}$.
\end{thm}

\begin{proof}
 Notice that the functor $\Hom_R(B^\bu,{-})$ takes complexes of
$S$\+torsion $R$\+modules to complexes of $S$\+contramodule
$R$\+modules (by Lemma~\ref{hom-contramodule}), and the functor
$B^\bu\ot_R{-}$ takes complexes of $S$\+contramodule $R$\+modules
to complexes of $S$\+torsion $R$\+modules (obviously).
 The construction of the two derived functors
$$
 \boR\Hom_R(B^\bu,{-})\:\sD^\st(R\modl_{S\tors})\lrarrow
 \sD^\st(R\modl_{S\ctra})
$$ and
$$
 B^\bu\ot_R^\boL{-}\:\sD^\st(R\modl_{S\ctra})\lrarrow
 \sD^\st(R\modl_{S\tors})
$$
is explained in~\cite[Appendix~B]{Pmgm}; see also~\cite[proofs of
Theorems~4.9 and~5.10]{Pmgm}.
 They form a pair of adjoint triangulated functors between
the derived categories $\sD^\st(R\modl_{S\tors})$ and
$\sD^\st(R\modl_{S\ctra})$ \cite[Appendix~B]{Pmgm}.

 As explained further in~\cite[proofs of Theorems~4.9 and~5.10]{Pmgm},
showing that the adjunction morphisms for this pair of adjoint
functors are isomorphisms in $\sD^\st(R\modl_{S\tors})$ and
$\sD^\st(R\modl_{S\ctra})$ reduces to the cases of a single injective
object $E\in R\modl_{S\tors}$ or a single projective object
$P\in R\modl_{S\ctra}$.
 The next proposition claims that the required morphisms are
quasi-isomorphisms.
\end{proof}

\begin{prop}
\textup{(a)} Let $E$ be an injective object in $R\modl_{S\tors}$ and
let $P^\bu$ be a bounded above complex of projective objects
in $R\modl_{S\ctra}$ endowed with a quasi-isomorphism of complexes
of $S$\+contramodule $R$\+modules $P^\bu\rarrow\Hom_R(B^\bu,E)$.
 Then the natural morphism of complexes of $S$\+torsion $R$\+modules
$B^\bu\ot_RP^\bu\rarrow E$ is a quasi-isomorphism. \par
\textup{(b)} Let $P$ be a projective object in $R\modl_{S\ctra}$ and
let $E^\bu$ be a bounded below complex of injective objects
in $R\modl_{S\tors}$ endowed with a quasi-isomorphism of complexes
of $S$\+torsion $R$\+modules $B^\bu\ot_R P\rarrow E^\bu$.
 Then the natural morphism of complexes of $S$\+contramodule
$R$\+modules $P\rarrow\Hom_R(B^\bu,E^\bu)$ is a quasi-isomorphism.
\end{prop}

\begin{proof}
 Part~(a): by Lemma~\ref{t-s-nakayama}(a), it suffices to show that
the morphism of complexes of $S$\+torsion $R$\+modules
$$
 B_s^\bu\ot_RP^\bu=T_s^\bu\ot_R B^\bu\ot_R P^\bu
 \lrarrow T_s^\bu\ot_RE
$$
is a quasi-isomorphism for every element $s\in S$.
 Notice that a morphism of complexes of $S$\+torsion $R$\+modules
is a quasi-isomorphism of complexes in the abelian category
$R\modl_{S\tors}$ if and only if it is a quasi-isomorphism of
complexes of $R$\+modules.
 According to Lemma~\ref{projective-tensor-lemma}(b), the morphism
$$
 T_s^\bu\ot_R P^\bu\lrarrow B_s^\bu\ot_RP^\bu
$$
induced by the morphism $b_s\:T_s^\bu\rarrow B_s^\bu$ is
a quasi-isomorphism of complexes of $R$\+modules.
 Since $T_s^\bu$ is a finite complex of (finitely generated) projective
$R$\+modules, the morphism of complexes of ($S$\+contramodule)
$R$\+modules
$$
 T_s^\bu\ot_RP^\bu\lrarrow T_s^\bu\ot_R\Hom_R(B^\bu,E)
$$
induced by the quasi-isomorphism $P^\bu\rarrow\Hom_R(B^\bu,E)$ is also
a quasi-isomorphism.
 We have a commutative diagram
$$
\begin{diagram}
\node{T_s^\bu\ot_RP^\bu}\arrow{s}\arrow{e}
\node{T_s^\bu\ot_R\Hom_R(B^\bu,E)} \arrow{s} \\
\node{T_s^\bu\ot_RB^\bu\ot_RP^\bu}\arrow{e}\node{T_s^\bu\ot_RE}
\end{diagram}
$$
 Hence it remains to show that the morphism of complexes
of $R$\+modules
$$
 T_s^\bu\ot_R\Hom_R(B^\bu,E)\lrarrow T_s^\bu\ot_R E
$$
induced by the morphism $b_s\:T_s^\bu\rarrow B_s^\bu=T_s^\bu\ot_R B^\bu$
is a quasi-isomorphism.
 Since $T_s^\bu\simeq\Hom_R(T_s^\bu,R)[-1])$, this is equivalent to
saying that the morphism
$$
 \Hom_R(b_s,E)\:\Hom_R(T_s^\bu\ot_RB^\bu,\>E)\lrarrow
 \Hom_R(T_s^\bu,E)
$$
is a quasi-isomorphism, which is the result of
Lemma~\ref{injective-hom-lemma}(b).

 Part~(b): by Lemma~\ref{t-s-nakayama}(b), it suffices to show that
the morphism of complexes of $S$\+contramodule $R$\+modules
$$
 \Hom_R(T_s^\bu,P)\lrarrow\Hom_R(T_s^\bu\ot_RB^\bu,\>E^\bu)=
 \Hom_R(B_s^\bu,E^\bu)
$$
is a quasi-isomorphism for every element $s\in S$.
 Notice that a morphism of complexes of $S$\+contramodule $R$\+modules
is a quasi-isomorphism of complexes in the abelian category
$R\modl_{S\ctra}$ if and only if it is a quasi-isomorphism of
complexes of $R$\+modules.
 According to Lemma~\ref{injective-hom-lemma}(b),
the morphism
$$
 \Hom_R(B_s^\bu,E^\bu)\lrarrow\Hom_R(T_s^\bu,E^\bu)
$$
induced by the morphism $b_s\:T_s^\bu\rarrow B_s^\bu$ is
a quasi-isomorphism of complexes of $R$\+modules.
 Since $T_s^\bu$ is a finite complex of (finitely generated) projective
$R$\+modules, the morphism of complexes of ($S$\+torsion) $R$\+modules
$$
 \Hom_R(T_s^\bu,\>B^\bu\ot_RP)\lrarrow\Hom_R(T_s^\bu,E^\bu)
$$
induced by the quasi-isomorphism $B^\bu\ot_RP\rarrow E^\bu$ is also
a quasi-isomorphism.
 We have a commutative diagram
$$
\begin{diagram}
\node{\Hom_R(T_s^\bu,P)}\arrow{s}\arrow{e}
\node{\Hom_R(T_s^\bu\ot_RB^\bu,\>E^\bu)} \arrow{s} \\
\node{\Hom_R(T_s^\bu,\>B^\bu\ot_RP)}\arrow{e}
\node{\Hom_R(T_s^\bu,E^\bu)}
\end{diagram} 
$$
 Hence it remains to show that the morphism of complexes of $R$\+modules
$$
 \Hom_R(T_s^\bu,P)\lrarrow\Hom_R(T_s^\bu,\>B^\bu\ot_RP)
$$
induced by the morphism $b_s\:T_s^\bu\rarrow B_s^\bu=T_s^\bu\ot_RB^\bu$
is a quasi-isomorphism.
 Since $\Hom_R(T_s^\bu,R)\simeq T_s^\bu[1]$, this is equivalent to saying
that the morphism
$$
 T_s^\bu\ot_RP\lrarrow T_s^\bu\ot_RB^\bu\ot_RP
$$
is a quasi-isomorphism, which is the result of
Lemma~\ref{projective-tensor-lemma}(b).
\end{proof}

\end{document}